\documentclass[final,onefignum,onetabnum]{siamart190516}

\newif\ifmarkup
\markupfalse

\newcommand{\markupadd}[1]{%
\ifmarkup
\textcolor{blue}{#1}%
\else
#1%
\fi
}

\usepackage{lipsum}
\usepackage{amsfonts}
\usepackage{graphicx}
\usepackage{epstopdf}
\usepackage{algorithmic}
\ifpdf
  \DeclareGraphicsExtensions{.eps,.pdf,.png,.jpg}
\else
  \DeclareGraphicsExtensions{.eps}
\fi

\headers{Cyclic Coordinate Dual Averaging with Extrapolation }{C.~Song, and J.~Diakonikolas }

\title{Cyclic Coordinate Dual Averaging\\ with Extrapolation%  for Generalized Variational Inequalities
\thanks{Submitted to the editors 01/07/2022.
\funding{This material is based upon research supported by, or in part by, the U.\ S.\ Office of
Naval Research under award number N00014-22-1-2348. The work was also partially supported by the NSF grant  DMS-2023239 and by a UW-Madison startup grant.}}}
 
\author{Chaobing Song\thanks{Department of Computer Sciences, UW-Madison, Madison, WI, USA (\email{chaobing.song@wisc.edu})}
\and Jelena Diakonikolas  \thanks{Department of Computer Sciences, UW-Madison, Madison, WI, USA (\email{jelena@cs.wisc.edu})}
}

\usepackage{amsopn}

\ifpdf
\hypersetup{
  pdftitle={Cyclic Coordinate Dual Averaging with Extrapolation},
  pdfauthor={C. Song and J. Diakonikolas}
}
\fi

\usepackage{amsmath,amsfonts,bm}
\usepackage{dsfont}
\usepackage{mathtools}
\usepackage{cite}

 \usepackage{color}
\usepackage{thm-restate}

\usepackage{booktabs}       % professional-quality tables
\usepackage{nicefrac}       % compact symbols for 1/2, etc.
\usepackage{microtype}      % microtypography

\usepackage{algorithm,algorithmic}
\usepackage[linesnumbered,algo2e,ruled,vlined]{algorithm2e}
\usepackage{graphicx}
\usepackage{subfigure}
\newtheorem{lemma}{Lemma}

\newtheorem{remark}{Remark}

\newtheorem{assumption}{Assumption}

\newtheorem{example}{Example}

\def\1{\bm{1}}

\def\vzero{{\bm{0}}}

\def\va{{\bm{a}}}
\def\vb{{\bm{b}}}
\def\vc{{\bm{c}}}

\def\vp{{\bm{p}}}
\def\vq{{\bm{q}}}

\def\vu{{\bm{u}}}

\def\vx{{\bm{x}}}
\def\vy{{\bm{y}}}
\def\vz{{\bm{z}}}

\def\vxt{\tilde{\bm{x}}}
\def\vxh{\hat{\bm{x}}}

\def\mA{{\bm{A}}}
\def\mB{{\bm{B}}}

\def\mF{{\bm{F}}}

\def\mI{{\bm{I}}}

\def\mQ{{\bm{Q}}}

\def\mQh{\widehat{\bm{Q}}}

\DeclareMathAlphabet{\mathsfit}{\encodingdefault}{\sfdefault}{m}{sl}
\SetMathAlphabet{\mathsfit}{bold}{\encodingdefault}{\sfdefault}{bx}{n}

\def\gB{{\mathcal{B}}}

\def\gF{{\mathcal{F}}}

\def\gS{{\mathcal{S}}}

\def\sR{{\mathbb{R}}}

\newcommand{\E}{\mathbb{E}}

\DeclareMathOperator*{\argmin}{arg\,min}

\newcommand{\innp}[1]{\left\langle #1 \right\rangle}
\usepackage[colorinlistoftodos]{todonotes}
\newcommand{\Gap}{\mathrm{Gap}}

\newcommand{\dom}{\mathrm{dom}}

\usepackage{psfrag}

\makeatletter
\newcommand{\subalign}[1]{%
  \vcenter{%
    \Let@ \restore@math@cr \default@tag
    \baselineskip\fontdimen10 \scriptfont\tw@
    \advance\baselineskip\fontdimen12 \scriptfont\tw@
    \lineskip\thr@@\fontdimen8 \scriptfont\thr@@
    \lineskiplimit\lineskip
    \ialign{\hfil$\m@th\scriptstyle##$&$\m@th\scriptstyle{}##$\hfil\crcr
      #1\crcr
    }%
  }%
}
\makeatother

\begin{document}
\maketitle
% REQUIRED
\begin{abstract}
Cyclic block coordinate methods are a fundamental class of optimization methods widely used in practice and implemented as part of standard software packages for statistical learning. Nevertheless, their convergence is generally not well understood and so far their good practical performance has not been explained by existing convergence analyses. In this work, we introduce a new block coordinate method that applies to the general class of variational inequality (VI) problems with monotone operators. This class includes composite convex optimization problems and convex-concave min-max optimization problems as special cases and has not been addressed by the existing work. The resulting convergence bounds match the optimal convergence bounds of full gradient methods, but are provided in terms of a novel gradient Lipschitz condition w.r.t.~a Mahalanobis norm. For $m$ coordinate blocks, the resulting gradient Lipschitz constant in our bounds is never larger than a factor $\sqrt{m}$ compared to the traditional Euclidean Lipschitz constant, while it is possible for it to be much smaller. Further, for the case when the operator in the VI has finite-sum structure, we propose a variance reduced variant of our method which further decreases the per-iteration cost and has better convergence rates in certain regimes. To obtain these results, we use a gradient extrapolation strategy that allows us to view a cyclic collection of block coordinate-wise gradients as one implicit gradient.
\end{abstract}

% REQUIRED

% REQUIRED

\section{Introduction}
%
%
%\markupdelete{Large-scale optimization problems are widespread in machine learning, signal processing, and operations research. 
%The ever-increasing scale of these problems renders standard first-order methods that rely on full gradient information impractical for many settings of interest. Fortunately, most of the standard  problems possess useful structure that makes them amenable to efficient optimization methods that only access partial gradient information at a time. A specific instance is the class of  (block)} 
{Block} coordinate methods, which rely on accessing only a subset of coordinates of the objective function (sub)gradient at a time{, are a fundamental class of methods frequently used in large-scale optimization settings}~\cite{wright2015coordinate,nesterov2012efficiency}. These methods have been very popular over the past decade, finding applications in areas such as feature selection in high-dimensional computational statistics  \cite{wu2008coordinate,friedman2010regularization, mazumder2011sparsenet}, empirical risk minimization in machine learning \cite{nesterov2012efficiency,zhang2015stochastic, lin2015accelerated,allen2016even,alacaoglu2017smooth,gurbuzbalaban2017cyclic,diakonikolas2018alternating}, and distributed computing \cite{liu2014asynchronous, fercoq2015accelerated,richtarik2016parallel}. 

{Block} 
{c}oordinate methods are classified according to the order in which (blocks of) coordinates are selected and updated~\cite{shi2016primer}, generally falling into  the three main categories: (i) greedy, or Gauss-Southwell, methods, which greedily select coordinates that lead to the largest progress (e.g., coordinates with the largest magnitude of the gradient, which maximizes progress in function value for descent-type methods), (ii) randomized methods, which select (blocks of) coordinates according to some probability distribution over the coordinate blocks, and (iii) cyclic methods, which update (blocks of) coordinates in a cyclic order. Although greedy methods can be quite effective, they are generally limited by the greedy selection criterion, which (except in some very specialized settings; see, e.g.,~\cite{nutini2015coordinate}) requires reading full first-order information, in each iteration. Thus, more attention has been given to randomized and cyclic methods.

From the aspect of theoretical guarantees, a major advantage of randomized coordinate methods (RCM) over cyclic variants has been the simplicity with which convergence arguments can be carried out. 
By sampling coordinates randomly with replacement, the expectation of a coordinate gradient is the full gradient, thus the analysis can be largely reduced to that of standard gradient descent. As a result, many variants of RCM with provable guarantees have been proposed for both convex minimization problems \cite{nesterov2012efficiency,lin2015accelerated,fercoq2015accelerated,diakonikolas2018alternating,allen2016even,hanzely2019accelerated,nesterov2017efficiency} and convex-concave min-max problems \cite{dang2014randomized, zhang2015stochastic, alacaoglu2017smooth,chambolle2018stochastic,tan2018stochastic,carmon2019variance,latafat2019new, fercoq2019coordinate, alacaoglu2020random,song2021variance}. 
The complexity of RCM as measured by the number of times full gradient information is accessed is no worse {(and often much better)} than that for full gradient first-order methods, making RCM suitable for high-dimensional settings. 
However, these guarantees are attained only in expectation or with high probability. %\markupdelete{Meanwhile, to sample the coordinates, randomized methods must involve generation of pseudorandom numbers from a certain probability distribution, which makes the implementation complicated and may dominate the cost if the coordinate update is cheap.} 
Furthermore, in practical tasks such as the training of deep neural networks, the strategy of sampling with replacement is seldom used due to reduced performance caused by not iterating over all the coordinates with high probability in one pass (while sampling without replacement achieves this with probability one) \cite{bottou2009curiously}. 

Compared to sampling with replacement, cyclically choosing coordinates or sampling without replacement (i.e., cyclically choosing coordinate blocks with their order determined according to a random permutation) appears more natural. In fact, cyclic coordinate methods (CCMs) often have better empirical performance than RCM \cite{beck2013convergence,chow2017cyclic,sun2019worst}. Due to their simplicity and empirical efficiency, CCMs have been the default approach in many well-known software packages for high-dimensional computational statistics such as  GLMNet \cite{friedman2010regularization} and SparseNet \cite{mazumder2011sparsenet}.  

However, CCM is much harder to analyze than RCM because it is highly nontrivial to establish a connection between the (cyclically selected) coordinate gradient and full gradient. As a result, compared to RCM, there are hardly any theoretical guarantees for CCM. %\markupdelete{In the seminal paper about RCM \cite{nesterov2012efficiency}, Nesterov has remarked that it is ``almost impossible to estimate the rate of convergence'' of cyclic coordinate descent in the general problem case.} 
However, some guarantees have been provided in the literature, albeit often under very restrictive assumptions such the isotonicity of the gradient~\cite{saha2013nonasymptotic} or with convergence rates that do not justify better empirical performance of CCM over RCM~\cite{beck2013convergence}. In particular, the iteration complexity result from~\cite{beck2013convergence} for a standard, gradient descent-type CCM applied to smooth convex optimization has \emph{linear dependence on the ambient dimension} (or the number of blocks in the block coordinate setting).  This linear dependence is expected, as the argument from~\cite{beck2013convergence} relies on treating the cyclical coordinate gradient as an approximation of full gradient of the current iterate. Further, such a dependence is unavoidable in the worst case~\cite{sun2019worst}, and much of the follow-up work to~\cite{beck2013convergence} has focused on either improving the dependence on other problem parameters (such as the Lipschitz constants) or on addressing structured classes of quadratic optimization problems~\cite{sun2015improved,li2017faster,hong2017iteration,wright2020analyzing,lee2019random,gurbuzbalaban2017cyclic}. 

%%% Commenting out below; not needed for this version
%The linear dependence on the number of coordinate blocks in the iteration complexity of CCM can be improved to square-root dependence by employing Nesterov acceleration as in~\cite{beck2013convergence}, which improves the convergence rate from $1/k$ to $1/k^2.$ 
% \cb{We may consdier some way to say in the acceleration case, our improvement is $m^{1/4}.$} 
%We are not aware of any existing method of the CCM-type that attains provably better dependence on the number of blocks.   %\cb{some intro for accelerated version of CCM.  There also exists some attempt to get an accelerated version of CCM, which maintains the same dependence on $\epsilon$ with accelerated gradient descent, but still has linear dependence on the number of blocks and involves an extra step of computing full gradient per-iteration.    }

Beyond the setting of smooth convex optimization, \cite{chow2017cyclic} has provided convergence results for a variant of CCM applied to unconstrained monotone variational inequality problems (VIPs), where the operator $\mF: \sR^d \to \sR^d$ is assumed to be cocoercive. Cocoercivity is a very strong assumption, which leads to an equivalence between solving the original VIP (equivalently, finding a zero of $\mF$, which is also known as the monotone inclusion problem) and finding a fixed point of a nonexpansive (1-Lipschitz) operator (see, e.g.,~\cite[Chapter~12]{facchinei2007finite}). This condition already fails to hold for bilinear matrix games, which is one of the most basic setups of min-max optimization.   
%Moreover, the convergence rate of {$1/k^{1/4}$ for reducing $\|\mF(\vx)\|$} from~\cite{chow2017cyclic} is unsatisfying, as we expect faster convergence for this class of methods \cb{not very good.}.

Finally, variance reduction strategies have been widely used to reduce the per-iteration cost of optimization methods in large-scale finite-sum settings. {Apart from \cite{hamedani2020stochastic} which provided a variance reduced scheme for the special subclass of alternating minimization methods (i.e., cyclic methods with two blocks), to the best of our knowledge, prior to our work there existed no variance reduced schemes with provable complexity guarantees for cyclic methods with an arbitrary number of blocks.} %\markupdelete{For methods with cyclic schemes, the existing literature has almost exclusively focused on variance reduction strategies for shuffled SGD~\cite{park2020linear,gurbuzbalaban2017convergence,mokhtari2018surpassing,malinovskybetter,shamir2016without,ying2020variance}, which updates all the coordinates in each iteration but accesses each sample function in a cyclic way. In the context of cyclically updating coordinates, to the best of our knowledge, there exists only one very recent variance reduced primal-dual method \cite{hamedani2020stochastic} which considers two-block cyclic coordinate update (i.e., first updating \emph{all} the primal coordinates, then updating \emph{all} the dual ones). However, there had been no work on variance reduction for general cyclic coordinate updates.} % had remained open. 

In summary, prior to this work, the following questions had remained open:
\begin{enumerate}
    \item Is it possible to develop a {CCM} method  that has a better dependence on the number of blocks{, even in the special case of smooth convex optimization}?
    \item Is it possible to obtain convergence guarantees for a CCM applied to the general class of variational inequality problems?
    \item Is it possible to further reduce per-iteration cost and improve  {overall} complexity results using variance reduction?
\end{enumerate}

%As a result, the following problems have remained open: (i) It is not known whether the linear dimension dependence of CCM can be improved even for smooth convex optimization problems; and (ii) It is not known whether CCM can have convergence guarantees for general monotone VIPs;\cb{(iii) acceleration}. As monotone VIPs include convex minimization problems as a special case, in this paper, we address the two questions by studying a new CCM-type method for monotone VIPs with strong convergence guarantees. \cb{Here need to change as A-CODER is added. }

%%%%%%%%%%%%%%%%%%%%%%%
\subsection{Our Contributions}
We consider generalized Minty variational inequality (GMVI) problems, which ask for finding $\vx^*$ such that 
\begin{equation}\label{eq:main-problem}\tag{P} 
\innp{\mF(\vx), \vx - \vx^*} + g(\vx) - g(\vx^*) \geq 0, \; \forall \vx \in \sR^d, 
\end{equation}
where $\mF:\sR^d \to \sR^d$ is a monotone Lipschitz operator and $g: \sR^d \to \sR\cup \{+\infty\}$ is a proper, extended-valued, convex, lower semicontinuous, block-separable function with an efficiently computable proximal operator (see Section~\ref{sec:prelims} for precise definitions). Our problem of interest~\eqref{eq:main-problem} captures broad classes of optimization problems, such as convex-concave min-max optimization\vspace{-1.5mm}
\begin{equation}\label{eq:problem-min-max}\tag{P\textsubscript{MM}}
    \min_{\vx^1 \in \sR^{d^1}}\max_{\vx^2 \in \sR^{d^2}} \Phi(\vx^1, \vx^2),
    \vspace{-1.5mm}
\end{equation}
where $\Phi(\vx^1, \vx^2):= \phi(\vx^1, \vx^2) + g^1(\vx^1) - g^2(\vx^2),$ $d^1 + d^2 = d,$ $\phi$ is convex-concave and smooth, and $g^1, g^2$ are convex and ``simple'' (i.e., have efficiently computable proximal operators), and convex composite optimization
\begin{equation}\label{eq:problem-comp-opt} \tag{P\textsubscript{CO}}
    \min_{\vx \in \sR^d} \big\{f(\vx) + g(\vx)\big\},
\end{equation}
where $f$ is smooth and convex, and $g$ is convex and ``simple''. 
\vspace{-1mm}
To reduce \eqref{eq:problem-min-max} to \eqref{eq:main-problem}, it suffices to stack vectors $\vx^1, \vx^2$ and define $\vx = (\vx^1, \vx^2),$ $\mF(\vx) = \big[\subalign{\nabla_{\vx^1} \phi(\vx^1, \vx^2)\\ - \nabla_{\vx^2} \phi(\vx^1, \vx^2)}\big],$ $g(\vx) = g^1(\vx^1) - g^2(\vx^2).$ To reduce \eqref{eq:problem-comp-opt} to \eqref{eq:main-problem}, it suffices to take $F(\vx) = \nabla f(\vx)$, while $g$ is the same for both problems. See, e.g.,~\cite{nemirovski2004prox,malitsky2019golden} and Corollaries~\ref{cor:comp-opt} and~\ref{cor:min-max} for more information. 

As is standard, we also assume that the operator $\mF$ admits a coordinate-friendly structure so that a full pass of cyclically computing (and updating) coordinate gradients has the same order of cost as computing the full gradient at a point.  
Our goal is to find an $\epsilon$-accurate solution to~\eqref{eq:main-problem} defined as $\vx^*_{\epsilon}$ that satisfies
\begin{equation}\label{eq:main-problem-eps}\tag{P\textsubscript{approx}}%$_\epsilon$} 
\Gap(\vx^*_{\epsilon};\vu) :=  \innp{\mF(\vx),  \vx^*_{\epsilon}-\vx} + g(\vx^*_{\epsilon}) - g(\vx) \le \epsilon, \; \forall \vx \in \gB.
\end{equation}
When the domain of $g$ is compact, we take $\gB = \dom(g).$ When the domain of $g$ is not compact, it is generally not possible to satisfy the inequality from \eqref{eq:main-problem-eps} for a finite non-negative $\epsilon,$ unless $\vx^*_{\epsilon}$ is an optimal solution. To see this, consider, for example, the case when $d$ is even, $g \equiv 0$ and $\mF = \begin{bsmallmatrix}&\vzero &\mI\\ &-\mI &\vzero\end{bsmallmatrix}$. Clearly, $\mF$ is monotone and \eqref{eq:main-problem} has a (unique) solution at $\vx^* = \vzero.$ However, for any $\vx^*_\epsilon \neq \vx^*$ it holds $\inf_{\vx \in \sR^d}\innp{\mF(\vx), \vx - \vx^*_\epsilon} = -\infty,$ and, thus, \eqref{eq:main-problem-eps} cannot be satisfied for any finite $\epsilon$. To deal with this issue, as is standard (see, e.g.,~\cite{chambolle2011first}), when the domain of $g$ is non-compact we will require that $\gB$ is compact (typically a ball of constant radius centered at an optimal solution). 
% \cb{This is not good. When the domain of $g$ is not compact, it is not possible to satisfy the above inequality for any $\vx. $ }

In the finite sum setting where we employ variance reduction, we use a slightly weaker notion of $\epsilon$-approximation, requiring
\begin{equation}\label{eq:main-problem-eps-vr}\tag{P\rlap{\textsuperscript{VR}}\textsubscript{approx}}%$_\epsilon$} 
\widehat{\Gap}(\vx^*_{\epsilon}; \vx) = \innp{\mF(\vx) + g'(\vx), \vx^*_{\epsilon} - \vx} \leq \epsilon, \; \forall \vx \in \gB,\, \forall g'(\vx) \in \partial g(\vx),
\end{equation}
where $\partial g(\vx)$ denotes the subdifferential set of $g$ at point $\vx$ and the same rules about choosing $\gB$ apply as for \eqref{eq:main-problem-eps}. 

To attain this goal for the general problem \eqref{eq:main-problem}, we propose the \emph{Cyclic cOordinate Dual avEraging with extRapolation (CODER)} method. To the best of our knowledge, this method is novel even in the setting of one\footnote{A method of mirror-descent style \cite{kotsalis2020simple}  shares a similar operator extrapolation idea.}  (i.e., in the full gradient setting) or two (i.e., in the primal-dual setting) blocks. 
Based on a novel Lipschitz condition for $\mF$ w.r.t.~a Mahalanobis norm that we introduce (see Assumption~\ref{assmpt:new-Lip} for a precise definition), in the general multi-block setting, CODER needs to access $O(\hat{L}/\epsilon)$\footnote{Here, for simplicity, we suppress the dependence on the diameter of the feasible set and/or initial distance to optimum, and instead focus on the dependence on $\hat{L}$ and $\epsilon.$ Precise bounds involving the dependence on all problem parameters are provided in Theorems~\ref{thm:main-coder} and \ref{thm:main-coder-vr}.} equivalent full gradients to construct an $\epsilon$-approximate solution, where $\hat{L}$ is the Lipschitz constant in Assumption~\ref{assmpt:new-Lip}. Moreover, if $g(\vx)$ is assumed to be $\gamma$-strongly convex ($\gamma>0$), the oracle complexity of CODER becomes $O\big(\frac{\hat{L}}{\gamma}\log \frac{1}{\epsilon}\big)$. Both complexity results are dimension independent under the Lipschitz condition we define. In terms of the connection with the more traditional Lipschitz constant $L$ of $\mF$ (see Assumption~\ref{ass:Lip}), we show that in general $\hat{L}\le \sqrt{m} L$, where $m$ is the number of coordinate blocks. {Thus, even for the special case of smooth convex optimization and using the worst-case bound for $\hat{L}$, this constitutes a $\sqrt{m}$ (or $\sqrt{d}$ for coordinate methods) improvement over the state of the art for cyclic methods and is the first improvement in terms of the dependence on $m$ in nearly ten years.} {Moreover}, the Lipschitz constant resulting from our analysis is often lower than the Euclidean Lipschitz constant (see Section~\ref{sec:prelims} for further discussion). 

Besides the improved complexity results stated above, to the best of our knowledge, our work is the first to provide any type of convergence guarantees for CCM methods applied to GMVI. Meanwhile, we provide a consistent analysis for the unconstrained/constrained/proximal settings\footnote{These settings correspond to the settings in which $g(\vx)$ is identical to zero, is the indicator function of a ``simple'' convex constrained set, or is  a ``simple'' convex function,  respectively, where ``simple'' means  that the function has efficiently computable projection/proximal operators). }, which is nontrivial for CCM methods \cite{beck2013convergence,chow2017cyclic}. 
Finally, our method applies to arbitrary block separation, which is highly nontrivial in the min-max setting, where vanilla CCM and RCM diverge in general (see Remark \ref{rem:non-convergence-of-vanilla-methods}).

To prove our main results, instead of treating coordinate gradient as an approximation of the full gradient, we consider a novel approximation strategy that relates the \emph{collection of cyclic coordinate gradients} from one full pass over the coordinates to a certain full \emph{implicit gradient}. 
This collection perspective helps us  {improve} the linear dependence on the dimension (or number of coordinate blocks).  %
To make our results applicable to GMVI problems, we introduce an extrapolation step on the operator, which is inspired by the very recent paper \cite{hamedani2018primal} that considered  non-bilinear convex-concave min-max problems, in the full gradient setting. %Such a strategy was also adopted by \cite{kotsalis2020simple}. 

Additionally, in the finite sum setting where $\mF$ can be expressed as $\mF(\cdot) = \frac{1}{n}\sum_{t=1}^n \mF_t(\cdot)$, we provide a variance reduced variant of CODER, VR-CODER, which reduces the per-iteration cost in the large data regimes and attains an overall improved complexity bound. To obtain this result, we combine the operator extrapolation with a double-loop variance reduction strategy that is most closely related to SVRG~\cite{johnson2013accelerating}; however, there are also important differences. The simple adding of the SVRG-style operator estimate to the extrapolated operator turns out to be insufficient to cancel out all the error terms in the analysis. For this reason, our stochastic extrapolated operator estimate also employs point extrapolation (see Step~\ref{step:vr-coder-extrapolation} in Algorithm~\ref{alg:coder-vr} and the corresponding discussion in Section~\ref{sec:vr-coder}). 

% \cb{The connection with VR-PDHG, extragradient, etc. }
VR-CODER can be viewed as a natural extension of the very recent results for  (one-block) variance reduced extragradient method \cite{alacaoglu2022stochastic}  and (two-block) variance reduced primal-dual hybrid gradient method \cite{hamedani2020stochastic} to multi-block. To attain this,  we conduct our convergence analysis based on both our novel definition of Lipschitz constants and the classical Euclidean Lipschitz constant.

% \vspace{-2mm}
%%%%%%%%%%%%%%%%%%%
\subsection{Related Work}
% \vspace{-2mm}
As discussed earlier, despite significant research activity devoted to randomized coordinate methods~\cite{nesterov2012efficiency,lin2015accelerated,fercoq2015accelerated,diakonikolas2018alternating,allen2016even,hanzely2019accelerated,nesterov2017efficiency,zhang2015stochastic,alacaoglu2017smooth,tan2018stochastic,song2021variance}, far less attention has been given to cyclic coordinate variants, and specifically to their rigorous convergence guarantees. 
In particular, while convergence guarantees have been established for smooth convex optimization problems in~\cite{beck2013convergence}, the obtained bounds exhibit at least linear dependence on the number of blocks (equal to the dimension in the coordinate case). Further, the bound from~\cite{beck2013convergence} also scales with $L_{\max}/L_{\min}$, where $L_{\max}$ and $L_{\min}$ are the maximum and the minimum Lipschitz constants over the blocks, which is unsatisfying, as (block) coordinate methods often exhibit improvements over full gradient methods when the Lipschitz constants over blocks are highly non-uniform. 
% \cb{This sentence may be not very strict. There exists some cases that $L_{\max}=L_{\min}\ll L.$}

In general, vanilla CCM is known to be order-$d^2$ slower than RCM in the worst case~\cite{sun2019worst}, where $d$ is the dimension, which is in conflict with its comparable and often superior performance in practice, as compared to RCM with the same step size strategy. 
% \cb{This sentence is only correct when we consider the same (adaptive) step size strategy for CCM and RCM. Meanwhile, the comparison is under the convex minimization setting.}.
This has led to more refined analyses of CCM with softer guarantees that explain why the worst-case examples are uncommon~\cite{gurbuzbalaban2017cyclic,lee2019random,wright2020analyzing}. However, the existing results only apply to unconstrained convex quadratic problems.

%% REMOVING THE PARAGRAPH BELOW SINCE IT IS NOT RELATED WORK.

%By contrast to existing work, we introduce a novel extrapolation-based CCM that applies to a broad class of generalized variational inequality problems, which contains (composite) convex  optimization as a special case. In the case of composite convex
%In the case of convex quadratic functions and unlike RCM or existing CCM methods, the results we obtain never exhibit worse complexity than the full gradient methods, and are often of much lower complexity.
%{Further, our method provably converges on min-max problems on which standard CCM and RCM methods diverge in general (see Remark~\ref{rem:non-convergence-of-vanilla-methods}).  }

% \vspace{-2mm}
%%%%%%%%%%%%%%%%%%%%%%%%%%%%% PRELIMS
%%
\section{Notation and New Lipschitz Conditions}\label{sec:prelims}
%%
% \vspace{-2mm}
We consider the $d$-dimensional Euclidean space $(\sR^d, \|\cdot\|),$ where $\|\cdot\| = \sqrt{\innp{\cdot, \cdot}}$ denotes the Euclidean norm, $\innp{\cdot, \cdot}$ denotes the (standard) inner product, and $d$ is assumed to be finite. Given a matrix $\mB,$ the operator norm of $\mB$ is defined in a standard way as $\|\mB\| = \sup\{\|\mB\vx\|: \vx\in\sR^d,\, \|\vx\|\le 1\}.$ We use $\vzero$ to denote an all-zeros vector with dimension determined by the context. Given a positive integer $m,$ let $[m]$ denote the set $\{1,2,\ldots, m\}.$
Throughout the paper, we assume that there is a given partition of the set $\{1, 2, \dots, d\}$ into sets $\gS^j$, $j \in \{1, \dots, m\},$ where $|\gS^j| = d^j > 0.$ For notational convenience, we assume that sets $\gS^j$ are comprised of consecutive elements from $\{1, 2, \dots, d\}$, that is,  $\gS^1 = \{1, 2, \dots, d^1\},$ $\gS^2 = \{d^1 + 1, d^1 + 2, \dots, d^1 + d^2\},\dots, \gS^m = \{\sum_{j=1}^{m-1}d^j + 1, \sum_{j=1}^{m-1}d^j + 2, \dots, \sum_{j=1}^{m}d^j\}$. This assumption is without loss of generality, as all our results are invariant to permutations of the coordinates. Given subvectors $\vx^j\in\sR^{d_j} (j\in[m])$, we use $(\vx^1, \vx^2,\ldots, \vx^m)$ to denote the long vector concatenating $\vx^j (j\in[m])$ orderly. For an operator $\mF: \sR^d \to \sR^d$, we use $\mF^j$ to denote its coordinate components indexed by $\gS^j.$ Given a sequence of positive semidefinite matrices $\{\mQ^j\}_{j=1}^m$, we define $\mQh^j$ by
\begin{equation}
  (\mQh^j)_{i, k} = 
\begin{cases}
        (\mQ^j)_{i, k}, & \text{ if } \min\{i, k\} > \sum_{\ell=1}^{j-1} d^{\ell},\\
        0, & \text{ otherwise. }
\end{cases}   \label{eq:Q-hat}
 \end{equation}
That is, $\mQh^j$ corresponds to the matrix $\mQ^j$ with the first $j - 1$ blocks of rows and columns set to zero. 

Given a proper, convex, lower semicontinuous function $g: \sR^d \to \sR \cup \{+\infty\},$ we use $\partial g(\vx)$ to denote the subdifferential set (the set of all subgradients) of $g$. Of particular interests to us are functions $g$ whose proximal operator (or resolvent), defined by
\begin{equation}\label{eq:prox-op}
    \mathrm{prox}_{\tau g}(\vu) := \argmin_{\vx \in \sR^d}\Big\{\tau g(\vx) + \frac{1}{2 }\|\vx - \vu\|^2\Big\}
\end{equation}
is efficiently computable for all $\tau > 0$ and $\vu \in \sR^d.$

To unify the cases in which $g$ are convex and strongly convex respectively, we say that $g$ is $\gamma$-strongly convex for $\gamma \geq 0,$ if for all $\vx, \vy \in \sR^d$ and $g'(\vx)\in \partial g(\vx)$, 
\begin{equation*}
    g(\vy) \geq g(\vx) +  \innp{g'(\vx), \vy - \vx} + \frac{\gamma}{2}\|\vy - \vx\|^2. 
\end{equation*}

\paragraph{Standard assumptions} Before introducing our new Lipschitz condition for $\mF,$ we provide the following standard assumptions first. 
\begin{assumption} \label{assmpt:general}
There exists at least one $\vx^*$ that solves~\eqref{eq:main-problem}. 
\end{assumption}

% \cb{I think that using M for the classical Lip constant is not widely accepted. Meanwhile, we need introduce both $\mQh$ and $\mQt$. So, we can use $L$ for the classical one, and use $\hat{L}$ and $\tilde{L}$ for the new ones.  }
\begin{assumption}\label{ass:mono}
$\mF: \sR^d \to \sR^d$ is monotone: $\forall \vx, \vy, \innp{\mF(\vx) - \mF(\vy), \vx - \vy} \geq 0.$
\end{assumption}

\begin{assumption}\label{ass:Lip}
$\mF: \sR^d \to \sR^d$ is $L$-Lipschitz: $\forall \vx, \vy,$ $ \|\mF(\vx) - \mF(\vy)\| \leq L\|\vx - \vy\|. $
\end{assumption}

\begin{assumption}\label{assmpt:strongly-convex}
$g(\vx)$ is $\gamma$-strongly convex $(\gamma\ge 0)$, block-separable over $\{\gS^j\}_{j=1}^m:$ $g(\vx) = \sum_{j=1}^m g^j(\vx^j)$, and admits an efficiently computable proximal operator.   
\end{assumption}

\markupadd{We note here that the assumption that $g(\vx)$ admits an efficiently computable proximal operator (in Assumption~\ref{assmpt:strongly-convex}) is made to ensure that the iterations of our algorithms (which make calls to a proximal operator for $g$) are not computationally expensive. However, the number of iterations of our algorithms (or the number of calls to the proximal operator for $g$) are bounded irrespective of this assumption.}

\paragraph{New Lipschitz condition} To obtain improved complexity results for CCMs, we introduce a novel Lipschitz condition w.r.t.~a Mahalanobis norm. 

\begin{assumption}\label{assmpt:new-Lip} 
There exists a sequence of positive semidefinite matrices $\{\mQ^j\}_{j=1}^m$ such that each $\mF^j(\cdot)$ is $1$-Lipschitz continuous w.r.t.~the norm $\|\cdot\|_{\mQ^j},$ i.e., $\forall \vx, \vy \in \sR^d,$
\begin{equation}\label{eq:block-Lipschitz}
\|\mF^{j}(\vx) - \mF^{j}(\vy)\| \le \sqrt{(\vx-\vy)^T\mQ^j(\vx-\vy)} = \|\vx - \vy\|_{\mQ^j},
\end{equation}
where $\mF^j(\vx)$ is the $d^j$-dimensional subvector comprised of the $\gS^j$ coordinates of $\mF(\vx).$ 
Further, $\sqrt{\Big\|\sum_{j=1}^m {\mQh}^j \Big\|}  = \hat{L} < \infty,$ where $\mQh^j$ is defined in Eq.~\eqref{eq:Q-hat}.  
%\end{equation}

% Meanwhile, $\sqrt{\Big\|\sum_{i=1}^m {\mQt}^i \Big\|}  = \tilde{L}^2 < \infty,$ where $\mQt^i$ is defined by 
% $$
%     (\mQt^i)_{j, k} = 
%     \begin{cases}
%             (\mQ^j)_{j, k}, & \text{ if } \max\{j, k\} \le \sum_{\ell=1}^{i-1} d^{\ell},\\
%             0, & \text{ otherwise. }
%     \end{cases}
% $$
% That is, $\mQt^i$ corresponds to the matrix $\mQ^j$ with  the last $m-i+1$ blocks of rows and columns set to 0. 
% \cb{Give the definition of $\tilde{\mQ}$ here. A-CODER will use $\tilde{\mQ}$.  }
\end{assumption}
Note that if $\mF$ satisfies Assumption \ref{ass:Lip}, then Assumption \ref{assmpt:new-Lip}  can be trivially satisfied with $\mQ^j = L^2 \mI\; (j\in[m]),$ where $\mI$ is the identity matrix, as $\|\mF^j(\vx) - \mF^j(\vy)\|^2 \leq \|\mF(\vx) - \mF(\vy)\|^2 \leq L^2 \|\vx - \vy\|^2.$ However, choosing more general matrices $\mQ^j$ allows for more flexibility in adapting to the problem geometry.  In the following, let $\mA = [\va^1, \va^2, \ldots, \va^d] \in\sR^{n\times d}$ be a data matrix and $\vb\in\sR^n$ be the vector of corresponding labels. We now provide some concrete example applications for the setting of $\mQ^j.$
\begin{example}[Elastic net]\label{ex:elastic-net}
The elastic net problem $\min_{\vx \in \sR^d}\frac{1}{2}\|\mA\vx - \vb\|^2 + \lambda_1 \|\vx\|_1 + \frac{\lambda_2}{2}\|\vx\|^2$ ($\lambda_1\ge 0, \lambda_2 \ge 0$) is an example  of~\eqref{eq:problem-comp-opt} and a special case of~\eqref{eq:main-problem}, where $\mF(\vx) = \mA^T(\mA\vx - \vb)$, $g(\vx) = \lambda_1 \|\vx\|_1 + \frac{\lambda_2}{2}\|\vx\|^2$. Observe that when $\lambda_1 = 0,$ $\lambda_2 > 0,$ elastic net reduces to ridge regression, while when $\lambda_1 >0,$ $\lambda_2 = 0,$ the problem reduces to LASSO. For this setup, we have $\|\mF(\vx) - \mF(\vy)\| = \|\mA^T\mA(\vx-\vy)\| = \sqrt{(\vx-\vy)^T (\mA^T\mA)^2(\vx-\vy)}.$ The tightest Lipschitz constant of $\mF(\vx)$ that we can select is $L = \|\mA^T\mA\|.$ Meanwhile,  letting $m=d$, $d^1 = d^2 = \cdots = d^m =1,$ we have $\|\mF^j(\vx) - \mF^j(\hat{\vx})\| = \|(\va^j)^T\mA(\vx-\hat{\vx})\| = \sqrt{(\vx-\hat{\vx})^T {\mQ}^j (\vx-\hat{\vx})}$ with  ${\mQ}^j = \mA^T \va^j (\va^j)^T \mA.$ 
\end{example}

%\cb{Some problem exists here. In classical version, SVM is $\frac{1}{n}\sum\{1 - b_i \va_i^T \vx, 0\}.$}
%\cb{definition for indicator function $\mathds{1}.$ }
% \vspace{-1.2mm}
\begin{example}[$\ell_1$ regularized SVM]\label{exam:svm}
The $\ell_1$-norm regularized support vector machine (SVM) is $\min_{\vx\in\sR^d}\{\max\{\mathbf{1}-\bar{\mA} \vx, \vzero\}  + \lambda \|\vx\|_1 \},$ where $\bar{\mA} = [b_1\va_1, b_2\va_2, \ldots, b_d \va_d]$, % with $\vb\in \{1, -1\}^n$ and $\circ$ denoting the element-wise Hadamard product,  
$\lambda\ge 0,$ and $\max\{\cdot,\cdot\}$ is applied in an element-wise way. 
Observing that $\max\{1-x, 0\} = \max_{-1\le y\le 0} (x - 1)y$, it follows that the SVM problem is an instance of \eqref{eq:problem-min-max} with $\mF(\vx, \vy) = (\bar{\mA}^T\vy,  -(\bar{\mA}\vx - \mathbf{1}))$ and  $g(\vx,\vy) = \lambda \|\vx\|_1 +  \sum_{j=1}^n \mathds{1}_{ -1\le y_j \le 0},$ where $\mathds{1}_{ -1\le y_j \le 0}$ is the convex indicator function of the interval $[-1, 0]$ (zero within the interval and infinite outside it). The tightest Lipschitz constant of $\mF(\vx)$ that we can select is $L = \|\bar{\mA}\|.$ Let $m = d+n, d^1 = d^2 = \cdots = d^{d+n} = 1,$ $\bar{\mA} = (\bar{\va}_1, \bar{\va}_2, \ldots, \bar{\va}_d) = (\bar{\vc}_1, \bar{\vc}_2, \ldots, \bar{\vc}_n)^T$. Then we have for $1\le j \le d$, $\|\mF^j(\vx, \vy) - \mF^j(\hat{\vx}, \hat{\vy})\| = \|\bar{\va}_j^T(\vy - \hat{\vy})\|$ with $\mQ^j =\begin{bsmallmatrix}&\vzero &\vzero \\ &\vzero &\bar{\va}_j \bar{\va}_j^T \end{bsmallmatrix}$ and for $d+1\le j \le d+n$, $\|\mF^j(\vx, \vy) - \mF^j(\hat{\vx}, \hat{\vy})\| = \|\bar{\vc}_j^T(\vx - \hat{\vx})\|$ with    $\mQ^j =\begin{bsmallmatrix} & \bar{\vc}_j\bar{\vc}_j^T  &\vzero \\ &\vzero &\vzero \end{bsmallmatrix}.$  
\end{example}  
% \cb{ It should be  $g(\vx, \vy) = \lambda\|\vx\|_1 -\sum_{j=1}^n \mathds{1}_{ -1\le y_j \le 0}$? }

% \vspace{-2.5mm}
\paragraph{Comparison of Lipschitz assumptions} 
Standard Lipschitz assumptions that are used for full gradient methods are typically stated as in Assumption \ref{ass:Lip}. %
Observe that the Lipschitz constant of the entire operator $\mF$ under our assumptions is bounded by $\sqrt{\|\sum_{j=1}^m \mQ^j\|},$ as, $\forall \vx, \vy,$
\begin{align*}
    \|\mF(\vx) - \mF(\vy)\|^2 = \sum_{j=1}^m \|\mF^j(\vx) - \mF^j(\vy)\|^2 &\leq \sum_{j=1}^m (\vx - \vy)^T \mQ^j (\vx - \vy)\\
    &\leq \Big\|\sum_{j=1}^m \mQ^j\Big\|\|\vx - \vy\|^2.
\end{align*}
In the worst case for full gradient methods, it is possible that $L = \sqrt{\big\|\sum_{j=1}^m \mQ^j\big\|}$, and this worst case  happens for many interesting examples discussed above.  The guarantees that we provide for our method CODER   are in terms of $\hat{L} = \sqrt{\big\|\sum_{j=1}^m \mQh^j\big\|}$. It is not hard to show that in general $\|\mQh^j\|\leq \|\mQ^j\|$. Thus, we have the following bound
% \begin{align*}
$\hat{L} \leq \sqrt{\big\|\sum_{j=1}^m \mQh^j\big\|} \leq \sqrt{\sum_{j=1}^m \big\|\mQh^j\big\|}
    \leq  \sqrt{\sum_{j=1}^m \big\|\mQ^j\big\|} \leq \sqrt{m}L. $
% \end{align*}
On the other hand, it is possible for $\hat{L}$ to be smaller than $L.$ %A simple example that demonstrates this is $\mQ^1 = \vu \vu^T,$ $\mQ^2 = \vv\vv^T,$ where $\vu^T = [1/t^2\; 1],$ $\vv^T = [-t\; 1/t]$ and $t\ge1$. As $\vu$ and $\vv$ are orthogonal,  $L =\sqrt{\|\mQ^1 + \mQ^2\|} =  \sqrt{\max\{\|\vu\|^2, \|\vv\|^2\}} = \sqrt{t^2 + \frac{1}{t^2}}.$ Further, $\mQh^1 = \mQ^1 = \begin{bsmallmatrix} 1/t^4 & 1/t^2\\ 1/t^2 & 1 \end{bsmallmatrix}$ and $\mQh^2 = \begin{bsmallmatrix}0 &0\\ 0 &1/t^2\end{bsmallmatrix}$. Thus,  $\hat{L} = \sqrt{\|\mQh^1 + \mQh^2\|} \leq \sqrt{\mathrm{Trace}(\mQh^1 + \mQh^2)}\leq \sqrt{1 + \frac{1}{t^2} +
%\frac{1}{t^4}}$. Now we can make $t$ arbitrarily large to get arbitrarily large $L/\hat{L}$.  
%

Using $\mQ^j = \mA^T \va^j (\va^j)^T \mA$, $1 \leq j \leq m,$ and $m=d$ in Example \ref{ex:elastic-net}, we compute the tightest constants  $L = \sqrt{\|\sum_{j=1}^m\mQ^j\|}$ and $\hat{L} = \sqrt{\|\sum_{j=1}^m\mQh^j\|}$ for the elastic-net problem on both simulated datasets and real datasets from the LibSVM library~\cite{CC01a}.   
As shown in  Table~\ref{tab:lip_consts} and Fig.~\ref{fig:lip-consts}, $\hat{L}$ is consistently lower than $L.$

% \cb{replace this figure.}
\begin{figure}
    \centering
    \hspace{\fill}
    \subfigure[$n = 200$]{\includegraphics[width=0.4\textwidth]{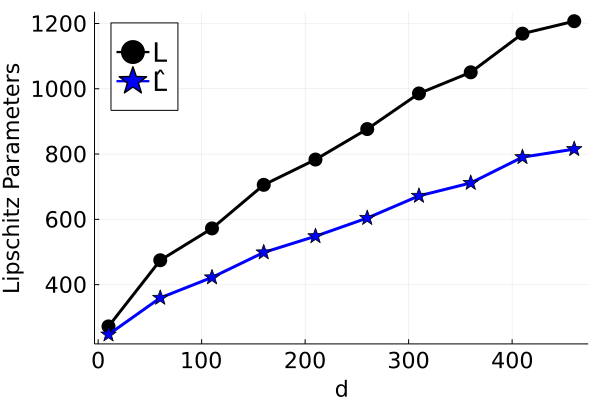}\label{fig:d}}
    \hspace{\fill}
    \subfigure[$d = 200$]{\includegraphics[width=0.4\textwidth]{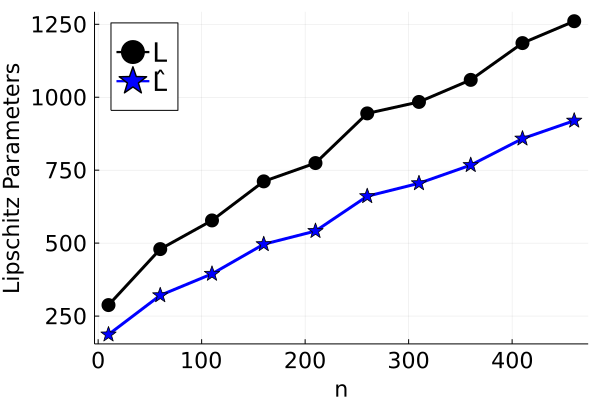}\label{fig:coor}}
    \hspace*{\fill}
    \caption{Lipschitz constants for elastic-net  problems  on  synthetic  datasets.
     All the samples are drawn i.i.d.~from standard multivariate Gaussian distribution. Fig.~\ref{fig:d} shows the values of Lipschitz parameters vs the dimension $d$ when $n$ is fixed to $200$; Fig.~\ref{fig:coor} shows the values of Lipschitz parameters vs number of samples $n$ when $d$ is fixed to 200. As we see, in both settings, the values of our  Lipschitz constant $\hat{L}$ are consistently lower than the classical Lipschitz constant $L$.
    }
    \label{fig:lip-consts}
\end{figure}
\begin{table}[th!]
\caption{Lipschitz constants for elastic-net problems on LibSVM datasets~\cite{CC01a}. To compute Lipschitz constants $L$ and $\hat{L}$, we normalize each sample vector to unit Euclidean norm.}
    \centering
    \begin{tabular}{ |c|c|c|c|c|c| } 
 \hline
 Dataset & a9a & australia & madelon & colon & mnist\\ 
 \hline
 $L$ & 15389.6   & 340.5  &   1992.4  &  9.0     &   24410.5   \\ 
  \hline
 $\hat{L}$ & 10358.8 & 238.1 & 1269.7    &   5.7    &   15236.1         \\ 
  \hline
%  $\tilde{L}$ & 9150.2 & 194.3  &  1267.1    &  5.6     &     15064.9    \\ 
%  \hline
\end{tabular}
    \label{tab:lip_consts}
\end{table}

In the literature on standard (randomized and cyclic) block coordinate methods and in the case where $\mF$ is the gradient of a convex function, the Lipschitz assumptions are typically stated as~:  
$
    \|\mF^j(\vx) - \mF^j(\vy)\| \leq L_j\|\vx - \vy\|,
$ 
where $\vx, \vy \in \sR^d$ are restricted to only differ over the $j^{\mathrm{th}}$ block of coordinates\cite{nesterov2012efficiency}. These assumptions are hard to directly compare to our Lipschitz assumptions stated in Assumption~\ref{assmpt:new-Lip}. What can be said is that in general $L_j \leq \|\mQ^j\|;$ however, note that our final convergence bound is in terms of $\big\|\sum_{j=1}^m \mQh^j\big\|,$ which is incomparable to weighted sums of $L_j$'s that typically appear in the convergence bounds for block coordinate methods. Further, note that the coordinate Lipschitz assumptions used for convex optimization are generally not compatible with min-max setups. In particular, for bilinear problems, all coordinate Lipschitz constants defined as in~\cite{nesterov2012efficiency} would be zero, which does not appear meaningful, given the non-zero complexity of bilinear problems~\cite{ouyang2018lower}. %\cb{We may have more direct explanation. }
\section{CODER Algorithm for Generalized Variational Inequalities}\label{sec:coder}
In this section, we provide the CODER algorithm (Algorithm \ref{alg:coder}) that applies to the general class of GMVI problems as stated in \eqref{eq:main-problem}, under Assumptions~\ref{assmpt:general}, \ref{ass:mono}, \ref{assmpt:strongly-convex}, and \ref{assmpt:new-Lip}.  

The general strategy can be summarized as follows. Let $\{a_k\}_{k\ge 0}$ and $\{A_k\}_{k\ge 0}$ be sequences of nonnegative numbers with $A_k =\sum_{i=1}^k a_i , a_0 = A_0 = 0.$ Let $\{\vx_k\}_{k\ge 0}$ be a sequence of points in $\dom(g)$ {generated by the algorithm and defined by \eqref{eq:x_k-def} below,} with the weighted average sequence $\{\tilde{\vx}_k\}_{k\ge 1}$ defined by $\tilde{\vx}_k = \frac{1}{A_k}\sum_{i=1}^k a_i \vx_i.$ Our goal is to show that $\Gap(\vxt_k; \vu)$ contracts at rate $1/A_k$ (with $A_k$ growing as fast as possible). Equivalently,  we need to show that $A_k\Gap(\vxt_k; \vu)$ is bounded above by a constant as $k$ increases. This strategy is inspired by~\cite{diakonikolas2019approximate,song2019unified};
 however, the concrete construction of the gap functions is novel and based on an operator extrapolation strategy that allows us to simultaneously handle an entire cyclic pass over the coordinate blocks. 

To attain this goal, we show that $A_k\Gap(\vxt_k; \vu)$ is  bounded above by a constant plus error terms. The error terms are correlated with the estimation sequence $\{\psi_k\}_{k\ge 1}$ satisfying 
% \begin{align}
% \psi_k(\vx) = - \sum_{i=1}^k a_i \big(\innp{\vq_i, \vx_i - \vx} + g(\vx_i) - g(\vu)\big) + \psi_0(\vu).   
% \end{align}
$\psi_k(\vx) = \sum_{j=1}^m \psi_k^j(\vx^j)$ with $\psi_k^j$ defined by $\psi_0^j(\vx^j) = \frac{1}{2}\|\vx^j - \vx_0^j\|^2$ and for $k \geq 1,$
\begin{equation}
    \psi_k^j(\vx^j) = \psi_{k-1}^j(\vx^j) + a_k\Big(\innp{\vq_k^j,\vx^j - \vx_k^j} + g^j(\vx^j) - g^j(\vx_k^j)\Big), \label{eq:es-psi-k-j}
\end{equation}
where in each $\psi_k^j(\vx^j)$, $\vx_k^j$ denotes the unique minimizer of  $\psi_k^j(\vx^j)$, i.e., 
\begin{equation}\label{eq:x_k-def}
    \vx_k^j = \argmin_{\vx^j \in \sR^{d^j}}\psi_k^j(\vx^j)
    % =  \mathrm{prox}_{A_k g^j}(\vx_0^j - \vz_k^j),
\end{equation}
and $\vq_k = (\vq_k^1, \vq_k^2,\ldots, \vq_k^m)$ is the extrapolated operator.
% where $\vz_k^j$ is recursively defined in Algorithm \ref{sec:coder}. \cb{check this. }
%
We remark that in the minimization problem defining $\vx_k^j$, terms $-a_k\Big(\big\langle\vq_k^j,  \vx_k^j\big\rangle + g^j(\vx_k^j)\Big)$ are treated as constants, so that 
$$\vx_k^j = \argmin_{\vx^j \in \sR^{|\gS^j|}}\Big\{\psi_{k-1}^j(\vx^j) + a_k \Big(\big\langle\vq_k^j,  \vx^j\big\rangle + g^j(\vx^j)\Big)\Big\}.$$
 Thus, under the definitions stated above, we have $\psi_k^j(\vx_k^j) = \psi_{k-1}^j(\vx_k^j).$ {Observe further that due to the block-separability of $\psi_k$, which is defined as $\psi_k(\vx) = \sum_{j=1}^m \psi_k^j(\vx^j)$, \eqref{eq:x_k-def} immediately implies that $\vx_k = \argmin_{\vx \in \sR^d}\psi_k(\vx).$} 
Meanwhile, due to the definition of $\psi_0$ and $\gamma$-strong convexity of $g(\vx)$, $\psi_k(\vx)$ is $(1+ A_{k-1}\gamma)$-strongly convex. 
Based on the above notation, we have the following bound. 
\begin{lemma}\label{lemma:gap-fn-construction}
{Let Assumptions~\ref{ass:mono} and \ref{assmpt:strongly-convex} hold and let $\{\vx_k\}_{k \geq 1}$ be a sequence of vectors obeying \eqref{eq:x_k-def}.} For any $\vu \in \dom(g)$ and any sequence of vectors $\{\vq_i\}_{1\leq i \leq k}$ in $ \sR^d,$  %we have
$$ A_k\Gap(\vxt_k; \vu) \le \sum_{i=1}^k E_i(\vu)  - \frac{1 + A_k \gamma}{2}\|\vu - \vx_k\|^2+ \frac{1}{2}\|\vu - \vx_0\|^2,$$ where the error sequence $\{E_i(\vu)\}_{1 \leq i \leq k}$ is defined by
\begin{align}
E_i(\vu) :=a_i  \innp{\mF(\vx_i) - \vq_i, \vx_i - \vu} - \frac{1+A_{i-1}\gamma}{2}\|\vx_i - \vx_{i-1}\|^2.\label{eq:coder-E_k-def}
\end{align}
% \begin{equation}\label{eq:eps_k-def}
% \begin{aligned}
%     \varepsilon_k(\vu) := \;&-\frac{1}{A_k}\psi_k(\vx_k)  + \frac{1}{A_k}\sum_{i=1}^k a_i \innp{\mF(\vx_i) - \vq_i, \vx_i - \vu}\\
%     &+ \frac{1}{A_k}\psi_0(\vu) -  \frac{1 + A_k \gamma}{2A_k}\|\vu - \vx_k\|^2,
% \end{aligned}
% \end{equation}
%

\end{lemma}
\begin{proof}
By the monotone property of $\mF$, we have $$\langle \mF(\vx), \vxt_k - \vx\rangle = \frac{1}{A_k}\sum_{i=1}^ka_i\langle \mF(\vx), \vx_i - \vx\rangle \le \frac{1}{A_k}\sum_{i=1}^ka_i \langle \mF(\vx_i), \vx_i - \vx\rangle.$$ Meanwhile, as $g$ is a convex function, by Jensen's inequality, $g(\vxt_k) \leq \frac{1}{A_k}\sum_{i=1}^k a_i g(\vx_i).$ So combining 
with the definition of $\Gap(\vxt_k; \vu)$ in \eqref{eq:main-problem-eps}, we have  
\begin{align}
  A_k    \Gap &(\vxt_k; \vu)
    \nonumber \\
   \leq \; & \sum_{i=1}^k a_i \big(\innp{\mF(\vx_i), \vx_i - \vu} + g(\vx_i) - g(\vu)\big) \nonumber\\
    =\; & \sum_{i=1}^k a_i \big(\innp{\vq_i, \vx_i - \vu} + g(\vx_i) - g(\vu)\big) -  \psi_0(\vu) \nonumber\\
    &+ \psi_0(\vu) +  \sum_{i=1}^k a_i \innp{\mF(\vx_i) - \vq_i, \vx_i - \vu} \nonumber\\
    =\; & -  \psi_k(\vu) +  \psi_0(\vu) + \sum_{i=1}^k a_i \innp{\mF(\vx_i) - \vq_i, \vx_i - \vu} \nonumber\\
     \le\; &  - \Big(\psi_k(\vx_k) + \frac{1 + A_k\gamma}{2}\|\vu - \vx_k\|^2\Big)  +\psi_0(\vu)
      +   \sum_{i=1}^k a_i \innp{\mF(\vx_i) - \vq_i, \vx_i - \vu},
     \label{eq:gap-1}  
    %  &\;  - \frac{1 + A_k\gamma}{2}\|\vu - \vx_k\|^2 +\psi_0(\vu) \nonumber\\ 
    %   \overset{(b)}{=}\;& \sum_{i=1}^k\Big( -(\psi_i(\vx_i) - \psi_{i-1}(\vx_{i-1})) + a_i  \innp{F(\vx_i) - \vq_i, \vx_i - \vu} \Big)  \nonumber\\
    %   &\;  - \frac{1 + A_k\gamma}{2}\|\vu - \vx_k\|^2 +\psi_0(\vu) \nonumber\\ 
    %   \overset{(c)}{\le}\;&  \sum_{i=1}^k \Big(- \frac{1+A_{i-1}\gamma}{2}\|\vx_i - \vx_{i-1}\|^2  + a_i  \innp{F(\vx_i) - \vq_i, \vx_i - \vu}\Big)   \nonumber\\
    %  &\;  - \frac{1 + A_k\gamma}{2}\|\vu - \vx_k\|^2 +\psi_0(\vu) \nonumber\\ 
    %  \overset{(d)}{=}\;&\sum_{i=1}^k E_i(\vu)  - \frac{1 + A_k\gamma}{2}\|\vu - \vx_k\|^2 +\psi_0(\vu),
\end{align}
where the last inequality is by the optimality condition of $\vx_k$ and $(1 + A_k\gamma)$-strong convexity of $\psi_k$, which leads to $\psi_k(\vu) \ge \psi_k(\vx_k) + \frac{1 + A_k\gamma}{2}\|\vu - \vx_k\|^2.$ 

On the other hand, using that $\psi_0(\vx_0) = 0$ and for $i \geq 1$, $\psi_{i-1}(\vx_i) - \psi_{i-1}(\vx_{i-1}) \geq \frac{1 + A_{i-1}\gamma}{2}\|\vx_i - \vx_{i-1}\|^2$ (as discussed in the paragraph above), we also have 
\begin{align}
 \psi_k(\vx_k) =\;&    \sum_{i=1}^k (\psi_i(\vx_i) - \psi_{i-1}(\vx_{i-1})) =  \sum_{i=1}^k (\psi_{i-1}(\vx_i) - \psi_{i-1}(\vx_{i-1})) \nonumber\\
 \ge\;& \sum_{i=1}^k \frac{1+A_{i-1}\gamma}{2}\|\vx_i - \vx_{i-1}\|^2. \label{eq:gap-2}
\end{align}
Combining Eqs.~\eqref{eq:gap-1} and \eqref{eq:gap-2}, and using the definitions of $E_i(\vu)$ and $\psi_0(\vu)$ completes the proof.
% The decomposition $\psi_k = \sum_{j=1}^m \psi_k^j$ follows by the block separability of all the components that comprise $\psi_k.$ 
% It remains to note that, as each $\psi_k^j$ is $(1 + A_k \gamma)$-strongly convex and minimized at $\vx_k^j,$ we have $\psi_k^j(\vu^j) \geq \psi_k^j(\vx_k^j) + \frac{1 + A_k\gamma}{2}\|\vu^j - \vx_k^j\|^2.$
\end{proof}

In Lemma \ref{lemma:gap-fn-construction}, for a fixed $\vu,$ $\psi_0(\vu)$ is a constant; meanwhile, the sequence $\{E_i(\vu)\}_{i\geq 1}$ denotes the error terms that need to be bounded above. In particular, if we prove that $\sum_{i=1}^k E_i(\vu)  \le \frac{1 + A_k\gamma}{2}\|\vu - \vx_k\|^2$, then %we can claim that  
$\Gap(\vxt_k; \vu)$ converges at rate $1/A_k$. 

Note that Lemma~\ref{lemma:gap-fn-construction} is generic---it applies to an arbitrary algorithm that satisfies its assumptions. However, bounding the error terms requires a more concrete algorithm. In the rest of the proof, we focus on analyzing the iteration complexity of CODER (Algorithm~\ref{alg:coder}). %To make the analysis more concrete, we provide the CODER description in Algorithm~\ref{alg:coder} and from now on focus on analyzing its iteration complexity. 
The main difference between CODER and a generic cyclic gradient method as stated in e.g.,~\cite{beck2013convergence} is the gradient extrapolation step defining $\vq_k^j$ in Step 7. In particular, if we had simply set $\vq_k^j$ as the $\vp_k^j$ in Step 6 of Algorithm~\ref{alg:coder},  we would recover the CCM from~\cite{beck2013convergence}. Observe further that, using the definition of the proximal step from Eq.~\eqref{eq:prox-op}, the definition of $\vx_k^j$ from Step~9 in CODER is equivalent to its definition in Eq.~\eqref{eq:x_k-def}. 
%%%%%%%%%%
\begin{algorithm*}[t!]
\caption{Cyclic cOordinate Dual avEraging with extRapolation (CODER)}\label{alg:coder}
\begin{algorithmic}[1]
\STATE \textbf{Input:} $\vx_{-1}=\vx_0\in\mathrm{dom}(g), \gamma \geq 0, \hat{L}>0, m, \{\gS^1, \dots, \gS^m\}$
\STATE \textbf{Initialization:} $\vp_0 = \mF(\vx_0), \vz_0 = \vzero$, $a_0= A_0 = 0$
\FOR{$k = 1$ to $K$} 
\STATE $a_k =  \frac{1+\gamma A_{k-1}}{2\hat{L}}, A_k = A_{k-1} + a_k$ 
\FOR{$j = 1$ to $m$} %
\STATE $\vp_k^j =   \mF^{j}(\vx^{1}_{k}, \ldots, \vx^{j-1}_{k}, \vx^{j}_{k-1},\ldots   \vx^{m}_{k-1})$
\STATE \label{step:coder-extrapolation} $\vq^j_k = \vp_k^j + \frac{a_{k-1}}{a_k}(\mF^j(\vx_{k-1}) - \vp_{k-1}^j)$
\STATE $\vz_k^j = \vz_{k-1}^j + a_k\vq^j_k$
\STATE \label{step:coder-xkj} $\vx_{k}^j = \mathrm{prox}_{A_k g^j}(\vx_0^j - \vz_k^j)$
\ENDFOR
\ENDFOR
\STATE \textbf{return} $\tilde{\vx}_K = \frac{1}{A_K}\sum_{k=1}^{K}a_k\vx_k$, $\vx_K$
\end{algorithmic}	
\end{algorithm*}
%%%%%%%%%%

%
\begin{lemma}\label{lemma:coder-gap-change}
Let $\{E_k(\vu)\}_{k\geq 1}$ be defined as in Lemma~\ref{lemma:gap-fn-construction} and let $\{\vx_k\}_{k \geq 1}$ evolve according to Algorithm~\ref{alg:coder}. Then{, under Assumption~\ref{assmpt:new-Lip},} $\forall k \ge 1,$ %it follows that:
\begin{align*}
 E_k(\vu) \leq \; & a_k \innp{\mF(\vx_k) - \vp_k, \vx_k - \vu} - a_{k-1} \innp{\mF(\vx_{k-1}) - \vp_{k-1}, \vx_{k-1} - \vu}\\
    &+ \frac{1 + A_{k-2}\gamma}{4}\|\vx_{k-1} - \vx_{k-2}\|^2 - \frac{1 + A_{k-1}\gamma}{4}\|\vx_{k-1} - \vx_k\|^2, 
\end{align*}
\end{lemma}
\begin{proof}
From Step 7 of Algorithm~\ref{alg:coder}, we have $\vq_k = \vp_k + \frac{a_{k-1}}{a_k}(\mF(\vx_{k-1}) - \vp_{k-1}).$ Then it follows that,
\begin{align}
    a_k \langle\mF(\vx_k) - \vq_k, &\,\vx_k - \vu\rangle\notag\\
    =\; & a_k \innp{\mF(\vx_k) - \vp_k, \vx_k - \vu} - a_{k-1} \innp{\mF(\vx_{k-1}) - \vp_{k-1}, \vx_k - \vu}\notag\\
    % =\; & a_k \innp{\mF(\vx_k) - \vp_k, \vx_k - \vu} - a_{k-1} \innp{\mF(\vx_{k-1}) - \vp_{k-1}, \vx_{k-1} - \vu}\notag\\
    % &+ a_{k-1} \innp{\mF(\vx_{k-1}) - \vp_{k-1}, \vx_{k-1} - \vx_k}\notag\\
    =\; & a_k \innp{\mF(\vx_k) - \vp_k, \vx_k - \vu} - a_{k-1} \innp{\mF(\vx_{k-1}) - \vp_{k-1}, \vx_{k-1} - \vu}\notag\\
    &+ a_{k-1}\innp{\mF(\vx_{k-1}) - \vp_{k-1}, \vx_{k-1} - \vx_k}. \label{eq:coder-inn-prod-term-1}
\end{align}
To simplify the notation, for all $k\ge 1, j\in[m]$, we define 
$$
\vy_{k, j} := (\vx^{1}_{k}, \ldots, \vx^{j-1}_{k}, \vx^{j}_{k-1},\ldots, \vx^{m}_{k-1}), 
$$
so that $\vp^j_k = \mF^j(\vy_{k, j}).$
Using the definition of $\vy_{k-1, j}$ and Young's inequality, we have, for all $j$ and all $\alpha > 0,$
\begin{align}
    \innp{\mF^j(\vx_{k-1}) - \vp_{k-1}^j, \vx_{k-1}^j - \vx_k^j} =\; & \innp{\mF^j(\vx_{k-1}) - \mF^j(\vy_{k-1, j}), \vx_{k-1}^j - \vx_k^j}\notag\\
    \leq\;& \frac{\alpha}{2}\|\mF^j(\vx_{k-1}) - \mF^j(\vy_{k-1, j})\|^2 + \frac{1}{2\alpha}\|\vx_{k-1}^j - \vx_k^j\|^2. \notag
\end{align}
By the definitions of $\mQh^j$ (in Eq.~\eqref{eq:Q-hat}), $\vy_{k-1,j}$ and Assumption \ref{assmpt:new-Lip}, we have 
\begin{align}
\|\mF^j(\vx_{k-1}) - \mF^j(\vy_{k-1, j})\|^2 \leq \|\vx_{k-1} - \vy_{k-1, j}\|_{\mQ^j}^2  = \|\vx_{k-1} - \vx_{k-2}\|_{\mQh^j}^2,
\end{align}
and then
\begin{equation}\notag
    \begin{aligned}
        \innp{\mF^j(\vx_{k-1}) - \vp_{k-1}^j, \vx_{k-1}^j - \vx_k^j} 
        \le\; & \frac{\alpha}{2}\|\vx_{k-1} - \vx_{k-2}\|_{\mQh^j}^2  + \frac{1}{2\alpha}\|\vx_{k-1}^j - \vx_k^j\|^2. 
    \end{aligned}
\end{equation}
Summing over $j$ and using the definition of the Lipschitz constant $\hat{L},$ we have
\begin{align}
\innp{\mF(\vx_{k-1}) - \vp_{k-1}, \vx_{k-1} - \vx_k}  =\; &  \sum_{j=1}^m\innp{\mF^j(\vx_{k-1}) - \vp_{k-1}^j, \vx_{k-1}^j - \vx_k^j} \notag\\
    \leq \; & \frac{\alpha}{2}\|\vx_{k-1} - \vx_{k-2}\|_{\sum_{j=1}^m\mQh^j}^2 
     + \frac{1}{2\alpha}\|\vx_{k-1} - \vx_k\|^2\notag\\
     \leq\; & \frac{\hat{L}^2\alpha}{2}\|\vx_{k-1} - \vx_{k-2}\|^2 + \frac{1}{2\alpha}\|\vx_{k-1} - \vx_k\|^2. \label{eq:coderF-upper-bound}
\end{align}
Combining with Eqs.~\eqref{eq:coder-E_k-def}, \eqref{eq:coder-inn-prod-term-1} and \eqref{eq:coderF-upper-bound},   we get
\begin{align*}
    E_k(\vu) \leq \; & a_k \innp{\mF(\vx_k) - \vp_k, \vx_k - \vu} - a_{k-1} \innp{\mF(\vx_{k-1}) - \vp_{k-1}, \vx_{k-1} - \vu}\\
    &+ \frac{a_{k-1} \hat{L}^2\alpha}{2}\|\vx_{k-1} - \vx_{k-2}\|^2 + \frac{a_{k-1}/\alpha - (1 + A_{k-1}\gamma)}{2}\|\vx_{k-1} - \vx_k\|^2, 
\end{align*}
and it remains to choose $\alpha = \frac{2 a_{k-1}}{1 + A_{k-1}\gamma}$ and use $a_{k-1}= \frac{1+A_{k-2}\gamma}{2\hat{L}}${, $A_{k-2} \leq A_{k-1}$}. 
\end{proof}
We are now ready to state and prove the main result of this section.
\begin{theorem}\label{thm:main-coder}
Let $\vx_0 \in \dom(g)$ be an arbitrary initial point and let $\{\vx_k\}_{k \geq 1}$ evolve according to Algorithm~\ref{alg:coder}. Then under Assumptions \ref{assmpt:general}, \ref{ass:mono}, \ref{assmpt:strongly-convex} and \ref{assmpt:new-Lip},
  $\forall k \geq 1$ and all $\vu \in \dom(g):$ 
\begin{equation}\label{eq:thm-main-gen-coder}
    \begin{aligned}
        A_k \Gap(\vxt_k; \vu) + \frac{1+ A_k\gamma}{4}\|\vu - \vx_k\|^2 \leq \frac{1}{2}\|\vu - \vx_0\|^2.
    \end{aligned}
\end{equation}
In particular, 
\begin{equation}\notag
\Gap(\vxt_k; \vu) \le\frac{1}{2 A_k}\|\vu - \vx_0\|^2.  
\end{equation}
Further, if $\vx^*$ is any solution to Problem \eqref{eq:main-problem}, we also have
\begin{equation}\notag
    \|\vx_k - \vx^*\|^2 \leq \frac{2}{1 + A_k\gamma}\|\vx_0 - \vx^*\|^2.    
\end{equation}
In both bounds, %
$A_k \ge \max\big\{\frac{k}{2\hat{L}},\,  \frac{1}{2\hat{L}}\big(1+\frac{\gamma}{2\hat{L}} \big)^{k-1}\big\}.$ 
\end{theorem}
\begin{proof}
Combining Lemmas \ref{lemma:gap-fn-construction}, \ref{lemma:coder-gap-change} and that, by initialization, $a_0 = 0$ and $\vx_0 = \vx_{-1} = \vzero,$ we have
\begin{align}
A_k\Gap(\vxt_k; \vu) 
\leq \;&  a_k \innp{\mF(\vx_k) - \vp_k, \vx_k - \vu} - \frac{1 + A_{k-1}\gamma}{4}\|\vx_{k-1} - \vx_k\|^2\nonumber\\
    &- \frac{1 + A_k \gamma}{2}\|\vu - \vx_k\|^2 
    + \frac{1}{2}\|\vu  - \vx_0\|^2.\notag
\end{align}
Following the same arguments as in the proof of Lemma~\ref{lemma:coder-gap-change}, we have, $\forall \alpha > 0,$
\begin{align*}
    \innp{\mF(\vx_k) - \vp_k, \vx_k - \vu} \leq \frac{\alpha \hat{L}^2}{2}\|\vx_k - \vx_{k-1}\|^2 + \frac{1}{2\alpha}\|\vx_k - \vu\|^2.
\end{align*}
Choosing $\alpha = \frac{2 a_{k}}{1 + A_{k}\gamma}$ and  by our setting, $a_{k} = \frac{1+A_{k-1}\gamma}{2\hat{L}}, A_k \ge A_{k-1},$ we get 
$$
    A_k\ \Gap(\vxt_k; \vu) \leq \frac{1}{2}\|\vu  - \vx_0\|^2 - \frac{1 + A_k \gamma}{4}\|\vu - \vx_k\|^2.
$$
% From Lemma~\ref{lemma:gap-fn-construction}, we have $\Gap(\vxt_k; \vu) \leq \frac{1}{A_k}\sum_{i=1}^k a_i \Gap(\vx_i; \vu) \leq \varepsilon_k(\vu)$, which gives the first two bounds. The third bound follows by $\Gap(\vx; \vx^*) \geq 0,$ $\forall \vx.$ 

Finally, as Algorithm~\ref{alg:coder} sets $a_i =  \frac{1+\gamma A_{i-1}}{2\hat{L}}, A_i = A_{i-1} + a_i,$ $\forall i \geq 1,$ we have $A_k \geq \frac{k}{2\hat{L}}$ (as $\gamma \geq 0$ and $A_0 = 0$) and $A_k \geq A_{k-1}\big(1 + \frac{\gamma}{2\hat{L}}\big) \geq A_1\big(1 + \frac{\gamma}{2\hat{L}}\big)^{k-1}$,  $\forall k \geq 1$.
\end{proof}
The implications of Theorem~\ref{thm:main-coder} on problems \eqref{eq:problem-comp-opt} and \eqref{eq:problem-min-max} are summarized in the following two corollaries. Here we only state the bounds for the optimality gap in Corollaries \ref{cor:comp-opt} and \ref{cor:min-max}, as the bounds on $\|\vx_k - \vx^*\|$ are immediate from Theorem~\ref{thm:main-coder}. Proofs are standard and are omitted for brevity. %Due to the limited space, we ignore the proofs of Corollaries \ref{cor:comp-opt} and \ref{cor:min-max} here. 
% \cb{I don't think that Corollary 1 is necessary as we give better algorithms in the next section.} \jd{Sure, but still nice to show that we can get results for \eqref{eq:problem-comp-opt} too from the general result.}
% \cb{OK.}
%
\begin{restatable}{corollary}{corpco}\label{cor:comp-opt}
Consider Problem~\eqref{eq:problem-comp-opt}, where the gradient of $f$ is $L$-Lipschitz in the context of Assumption~\ref{assmpt:general} %
and $g$ is $\gamma$-strongly convex for $\gamma \geq 0$, and let $\vx^* \in \argmin_{\vx} f(\vx) + g(\vx)$. If Algorithm~\ref{alg:coder} is applied to \eqref{eq:problem-comp-opt} with $\mF = \nabla f,$ then
$$
    f(\vxt_k) + g(\vxt_k) - (f(\vx^*) + g(\vx^*))\leq \frac{\|\vx^* - \vx_0\|^2}{2A_k}, 
$$
where $A_k \ge \max\big\{\frac{k}{2\hat{L}},\,  \frac{1}{2\hat{L}}\big(1+\frac{\gamma }{2\hat{L}} \big)^{k-1}\big\}.$
\end{restatable}
\begin{restatable}{corollary}{cormm}\label{cor:min-max}
Consider Problem~\eqref{eq:problem-min-max}, where $\phi$ is convex-concave and its gradient is $L$-Lipschitz in the context of Assumption~\ref{assmpt:general}, % 
and $g_1, g_2$ are $\gamma$-strongly convex ($\gamma \geq 0$ ) with compact domains. If Algorithm~\ref{alg:coder} is applied to~\eqref{eq:problem-min-max} with $\vx = \big[\substack{\vx^1\\ \vx^2}\big],$ $\mF(\vx) = \big[\subalign{\nabla_{\vx^1} \phi(\vx^1, \vx^2)\\ - \nabla_{\vx^2} \phi(\vx^1, \vx^2)}\big],$ and $g(\vx) = g^1(\vx^1) - g^2(\vx^2),$ then
$$
    \max_{\vy^2 \in \sR^{d^2}} \Phi(\vxt_{k}^1, \vy^2) - \min_{\vy^1 \in \sR^{d^1}} \Phi(\vy^1, \vxt_{k}^2) \leq \frac{D_1^2 + D_2^2}{2 A_k},
$$
where $D_1 = \sup_{\vx^1, \vy^1 \in \dom(g^1)}\|\vx^1 - \vy^1\|$, $D_2 = \sup_{\vx^2, \vy^2 \in \dom(g^2)}\|\vx^2 - \vy^2\|,$ and $A_k \ge \max\big\{\frac{k}{2\hat{L}},\,  \frac{1}{2\hat{L}}\big(1+\frac{\gamma }{2\hat{L}} \big)^{k-1}\big\}.$
\end{restatable}
\markupadd{\begin{remark}
    Convergence bound for $\Gap(\vxt_k; \vu)$ obtained in Theorem~\ref{thm:main-coder}, up to constants, is the same as the convergence bound that one would obtain from full vector update methods such as mirror-prox \cite{nemirovski2004prox} and dual extrapolation \cite{nesterov2007dual}, but with the full operator Lipschitz constant $L$ being replaced by our new Lipschitz constant $\hat{L}.$ As discussed before, $\hat{L}$ is never larger than $\sqrt{m}L,$ but it is usually smaller than $L$, due to our new block coordinate Lipschitz assumption (Assumption~\ref{assmpt:new-Lip} involving matrices $\mQ^j$, based on which $\hat{L}$ is defined)  that generally aligns better with the problem geometry. 
\end{remark}}
\begin{remark}\label{rem:non-convergence-of-vanilla-methods}
It seems natural to ask whether the extrapolation step in CODER (Line 7 in Algorithm~\ref{alg:coder}) is really needed or not. Let us refer to the cyclic and randomized coordinate method variants without the extrapolation step (i.e., with $\vq_k^i = \vp_k^i$ in Step 7 of CODER) as the proximal CCM (PCCM) and proximal RCM (PRCM). These methods only perform (block) coordinate dual-averaging steps (Step 9 of CODER). 
% Even though PCCM can be observed to perform well in the conducted experiments (see Section~\ref{sec:num-exp}), unlike CODER, neither PCCM nor PRCM are guaranteed to converge on the class of GMVI problems.
In particular, it is easy to construct examples on which both PCCM and PRCM diverge. Perhaps the simplest such example is the bilinear problem $\min_{\vx \in \sR^d}\max_{\vy \in \sR^d}\innp{\vx, \vy},$ where each pair $(x_i, y_i)$ is assigned to the same  block. In this case, the block coordinate updates of PCCM and PRCM boil down to (simultaneous) gradient descent-ascent updates, and, due to the separability of the objective function, the divergent behavior of both methods follows as a simple corollary of folklore results on the divergence of gradient descent-ascent  (see, e.g.,~\cite{salimans2016improved,liang2018interaction}). 
\end{remark}

\paragraph{Computational considerations}
At a first glance, it may seem like the usefulness of our method is limited by the parameter tuning required for constants $\hat{L}$ and $\gamma,$ which is a standard concern for most first-order methods, especially in the (block) coordinate setting. However, as we now argue, for most cases of interest this is not a concern. In particular, the strong convexity of $g$ typically comes from regularization, which is a design choice and as such is typically known. On the other hand, it turns out that for our approach to work, the knowledge of the Lipschitz parameter $\hat{L}$ is not required at all, as this parameter can be estimated adaptively using the standard doubling trick or a backtracking search as in, e.g.,~\cite{nesterov2015universal}. This can be concluded from the fact that the only place in the analysis where the Lipschitz constant of $\mF$ is used is to require $\|\mF(\vx_k) - \vp_k\| \leq \hat{L}\|\vx_k - \vx_{k-1}\|$, which allows a simple verification and update to $\hat{L}$ whenever the stated inequality is not satisfied. For completeness, in Appendix~\ref{appx:param-free}, we provide a parameter-free version of CODER.

\section{Variance Reduced CODER}\label{sec:vr-coder} 
In this section, we consider the GMVI problem in \eqref{eq:main-problem} with the finite-sum assumption $\mF(\vx) = \frac{1}{n}\sum_{t=1}^n \mF_t(\vx).$ For this problem, we show that we can incorporate a variance reduction strategy into the CODER algorithm and further reduce the per-iteration cost and improve the overall complexity result. The resulting algorithm, VR-CODER, is provided in Algorithm~\ref{alg:coder-vr}.  

VR-CODER combines an SVRG-type variance reduction with the CODER algorithm. It is a three-loop algorithm, there the outer two loops define the epochs and iterations of SVGR-type variance reduction in a standard way, and the innermost loop performs cyclic updates of CODER. To incorporate variance reduction, VR-CODER replaces the operator extrapolation step $\vq_k^j$ with a step that performs variance reduced operator extrapolation and, in addition, contains point extrapolation (Step~10 in Algorithm~\ref{alg:coder-vr}). In the definition of $\vq_{s, k}^j$ from Step~10, $\mF_t^{j}(\vy_{s,k,j})- \mF_t^{j}(\hat{\vx}_{s-1}) + \boldsymbol{\mu}_{s}^j$ is an unbiased estimate of $\mF^{j}(\vy_{s,k,j})$ of standard SVRG type, while $\mF_t^j(\vx_{s, k-1}) - \mF_t^{j}(\vy_{s,k-1,j})$ is an unbiased estimate of $\mF^j(\vx_{s, k-1}) - \mF^{j}(\vy_{s,k-1,j})$. Hence, $\mF_t^{j}(\vy_{s,k,j})- \mF_t^{j}(\hat{\vx}_{s-1}) + \boldsymbol{\mu}_{s}^j + \frac{a_{s,k-1}}{a_s}(\mF_t^j(\vx_{s, k-1}) - \mF_t^{j}(\vy_{s,k-1,j}) )$ provides an unbiased estimate of operator extrapolation step from CODER (Step~7  in Algorithm~\ref{alg:coder}). In addition to the unbiased estimate of CODER operator extrapolation, $\vq_{s, k}^j$ also contains point extrapolation (the $\beta(\vx_{s, k-1}^j - \vxh^j_{s-1})$ term). This point extrapolation is utilized in the analysis to bound the error terms. It is unclear whether the same results can be obtained without this point extrapolation. 
 
\begin{algorithm}[t!]
\caption{Variance Reduced Cyclic cOordinate Dual avEraging with extRapolation (VR-CODER)}\label{alg:coder-vr}
\begin{algorithmic}[1]
\STATE \textbf{Input:} $\vx_0 = \vx_{1,0}=\vxh_0\in\mathrm{dom}(g), \gamma \geq 0, L>0, \hat{L}>0, S, K, m, \{\gS^1, \dots, \gS^m\}$
\STATE \textbf{Initialization:} %$\vp_0 = \mF(\vx_0), 
$\vz_{1,0} = \vzero$, $a_0= A_0 = 0, a_1 = A_1 = \min\Big\{ \frac{\sqrt{K}}{8L},  \frac{K}{8\hat{L}}\Big\}, \beta =  \frac{2L}{\sqrt{K}}$
\FOR{$s=1$ to $S$}
    \STATE \label{step:vr-coder-mu} $\boldsymbol{\mu}_{s} = \mF(\hat{\vx}_{s-1})$ \hfill{\COMMENT{full epoch operator}}
    % \STATE $a_s =  \frac{1+\gamma A_{s-1}}{2{L}}, A_s = A_{s-1} + a_s$ 
    \STATE $a_{s, 0}  = a_{s-1},\; a_{s,1}=a_{s,2}=\cdots =a_{s,K} = a_s$
    \FOR{ $k=1$ to $K$}
         
        %  \STATE (How to initialize $\vp_{s, 0}$?)
        \FOR{ $j=1$ to $m$ }
            % \STATE $\mF^j_t(\vy_{s,k,j}) =   \mF_t^{j}(\vx^{1}_{s,k}, \ldots, \vx^{j-1}_{s,k}, \vx^{j}_{s, k-1},\ldots   \vx^{m}_{s, k-1})$
            \STATE $\vy_{s,k,j} = (\vx^{1}_{s,k}, \ldots, \vx^{j-1}_{s,k}, \vx^{j}_{s, k-1},\ldots   \vx^{m}_{s, k-1})$
            \STATE Choose $t$ in $[n]$ uniformly at random
            \STATE $\vq^j_{s,k} = \mF_t^{j}(\vy_{s,k,j})- \mF_t^{j}(\hat{\vx}_{s-1}) + \boldsymbol{\mu}_{s}^j + \frac{a_{s,k-1}}{a_s}(\mF_t^j(\vx_{s, k-1}) - \mF_t^{j}(\vy_{s,k-1,j}) ) +  \beta(\vx^j_{s, k-1} - \hat{\vx}_{s-1}^j)$ \hfill{\COMMENT{variance-reduced extrapolated operator}} \label{step:vr-coder-extrapolation}
            \STATE $\vz_{s,k}^j = \vz_{s,k-1}^j + a_s\vq^j_{s,k}$
            \STATE $\vx_{s,k}^j = \mathrm{prox}_{(A_{s-1} + \frac{a_s k}{K}) g^j}(\vx_0^j - \vz_{s,k}^j/K)   $
        \ENDFOR
    \ENDFOR
    \STATE \label{step:vr-coder-step-size} $a_{s+1} = \min\Big\{\big(1 + \frac{\gamma}{\beta}\big)a_s,   \, (1+A_{s}\gamma)\min\Big\{ \frac{\sqrt{K}}{8L},  \frac{K}{8\hat{L}}\Big\}\Big\}, A_{s+1} = A_s + a_{s+1}$
    \STATE $\hat{\vx}_s =\frac{1}{K}\sum_{k=1}^K\big( \frac{\beta }{\beta+\gamma}\vx_{s, k-1} + \frac{\gamma}{\beta+\gamma}\vx_{s,k}\big), \vxt_s =\frac{1}{K}\sum_{k=1}^K \vx_{s,k}$ 
    \STATE $\vy_{s+1,0,j} = \vy_{s, K, j}, \;\forall j\in[m]$
    \STATE $\vz_{s+1, 0} = \vz_{s, K},\, \vx_{s+1, 0} = \vx_{s, K},\, \vx_{s+1, -1} = \vx_{s, K-1}$
\ENDFOR
\STATE \textbf{return} $\frac{1}{A_S}\sum_{s=1}^S a_s \tilde{\vx}_s,$ $\vx_{S+1, 0}.$ 
\end{algorithmic}	
\end{algorithm}

% \cb{$\vx_{s+1,-1}$ has not been used in the algorithm, so we can just ignore it; but in the proof, we can define it. }

\paragraph{Assumptions for VR-CODER} To analyze VR-CODER, we assume that Assumptions~\ref{ass:mono} and \ref{assmpt:strongly-convex} hold for $\mF$ and $g$, respectively.  Meanwhile, we assume that Assumptions~\ref{ass:Lip} and \ref{assmpt:new-Lip} apply not only to $\mF$ but to all individual operators $\mF_t$ in the definition of $\mF.$  It is an open question whether any of these assumptions can be relaxed to only apply to $\mF.$ 

%\cb{The definition for filtration needs to be changed to $\gF_{s,k,j}$. }

%\cb{make assumptions further here.}

%\cb{give description for this algorithm; show how variance reduction performs.}

\iffalse
\begin{assumption}\label{ass:Lip-fs}
Each $\mF_t: \sR^d \to \sR^d$ is $L$-Lipschitz: $\forall \vx, \vy,$ $ \|\mF_t(\vx) - \mF_t(\vy)\| \leq L\|\vx - \vy\|.$
\end{assumption}

\begin{assumption}\label{assmpt:new-Lip-fs} 
There exists a sequence of positive semidefinite matrices $\{\mQ^j\}_{j=1}^m$ such that each $\mF_t^j(\cdot)$ is $1$-Lipschitz continuous w.r.t.~the norm $\|\cdot\|_{\mQ^j},$ i.e., $\forall \vx, \vy \in \sR^d,$
%
\begin{equation}\label{eq:block-Lipschitz-fs}
\|\mF_t^{j}(\vx) - \mF_t^{j}(\vy)\| \le \sqrt{(\vx-\vy)^T\mQ^j(\vx-\vy)} = \|\vx - \vy\|_{\mQ^j},
\end{equation}
%
where $\mF_t^j(\vx)$ is the $d^j$-dimensional vector comprised of the $\gS^j$ coordinates of $\mF_t(\vx).$ 
Further, $\sqrt{\Big\|\sum_{j=1}^m {\mQh}^j \Big\|}  = \hat{L} < \infty,$ where $\mQh^j$ is defined in \emph{Eq.}~\eqref{eq:Q-hat}.  
%\end{equation}
\end{assumption}

\fi

\paragraph{VR-CODER Convergence Analysis} 

Same as in the notation for $\mF$ in Section~\ref{sec:prelims}, we let $\mF_t^j$ denote its coordinate components indexed by $\gS^j$. 
For $s\ge 1$, $k\in[K]$ and $j\in[m]$, we define the generalized estimation sequence, 
\begin{equation}
\psi_{s,k}^j(\vx^j) = \psi_{s,k-1}^j(\vx^j) + a_s \Big(\big\langle\vq_{s,k}^j,  \vx^j - \vx_{s,k}^j\big\rangle + g^j(\vx^j) - g^j(\vx_{s,k}^j)\Big), \label{eq:es-def}
\end{equation}
where we assume that $\psi_{s, K}^j = \psi_{s+1,0}^j$ and $\psi_{1,0}^j(\vx^j) =  \frac{K}{2}\|\vx^j - \vx_0^j\|^2$. Meanwhile, we define $\psi_{s, k}(\vx) = \sum_{j=1}^m\psi_{s, k}^j(\vx^j)$ and $\vq_{s,k} = (\vq_{s,k}^1, \vq_{s,k}^2,\ldots, \vq_{s,k}^m)$. Then similar to Section~\ref{sec:coder}, from the definition of $\psi_{s,k}(\vx)$, we have that $\psi_{s,k}(\vx_{{s, }k}) = \psi_{s,k-1}(\vx_{{s, }k}).$  As $g(\vx) = \sum_{i=1}^m g^j(\vx^j)$ is assumed to be $\gamma$-strongly convex, it further follows that $\psi_{s,k}(\vx)$ for $0\le k < K,$ is $K(1+ A_{s-1}\gamma)$-strongly convex. 

The convergence analysis follows the same strategy as in Section~\ref{sec:coder}: our goal is to show that the appropriate gap  function (in this case the relaxed gap function $\widehat{\Gap}$) contracts at rate $1/A_k$, where $A_k$ grows as fast as possible. However, since VR-CODER is a randomized algorithm (due to using stochastic variance-reduced estimates of operator extrapolation), the gap function is bounded \emph{in expectation}. The process of bounding the error terms $E_{s, k}(\vu)$ appearing in the bound on the gap function, however, becomes much more technical than in the CODER analysis. For this reason, the proof of the main technical lemma (Lemma~\ref{lem:error-vr}) is deferred to Appendix~\ref{appx:proof-lemma-error-vr}.  

As before, we start with the generic bound on the gap function. 
%The estimate of the relaxed gap $\widehat{\Gap}(\tilde{\vx}_s; \vu)$ can then be bounded as follows. 

% Meanwhile, $\forall s\ge 1$, $k\in[K],j\in [m]\!$, let 
% \begin{equation}
% \vy_{s, k, j} = (\vx^{1}_{s,k}, \ldots, \vx^{j-1}_{s,k}, \vx^{j}_{s,k-1},\ldots   \vx^{m}_{s,k-1}) \label{eq:y-s-k-1}    
% \end{equation}
%  with $\vy_{s, K, j} = \vy_{s+1, 0, j}$. Then in Algorithm \ref{alg:coder-vr}, we have $\mF^j_t(\vy_{s,k,j}) = \mF_t(\vy_{s, k, j}).$ 

% \cb{Give the definition of $\vx_{s,k}$. }

% \cb{ Give the guarantee in terms of expectation-max.  }

% \cb{Introduce filtration. }

\begin{lemma}\label{lem:gap-vr}
{Let Assumptions~\ref{ass:mono} and \ref{assmpt:strongly-convex} hold. }In the $s$-th epoch {of Algorithm~\ref{alg:coder-vr}}, for all $\vu\in\dom(g)$, the gap function $\widehat{\Gap}(\vxt_s; \vu) $ w.r.t. $\vxt_s$ can be bounded as follows: for all $S\ge 1, K\ge 1,$ 
\begin{align}
 \sum_{s=1}^S a_s \widehat{\Gap}(\vxt_s; \vu)  \le\;&    \frac{1}{2}\|\vu -  \vx_{0}\|^2 \!-\!  \frac{1 + A_S \gamma}{2}\|\vu - \vx_{S+1,0}\|^2 +\frac{1}{K}\sum_{s=1}^S \sum_{k=1}^K E_{s,k}(\vu), \notag
\end{align} 
where 
\begin{align}
E_{s,k}(\vu) \!=\;&   a_s \innp{\mF(\vx_{s,k})\! - \!\!\vq_{s,k}, \vx_{s,k} - \vu} \!-\!  \frac{K(1+A_{s-1}\gamma)}{2}\|\vx_{s,k} - \vx_{s,k-1}\|^2 \nonumber\\
\;&   - \frac{a_s \gamma}{2}\|\vx_{s,k} - \vu\|^2. \notag
\end{align}
\end{lemma}
\begin{proof}
As $\mF$ is monotone and $g$ is $\gamma$-strongly convex, by the definition of $\tilde{\vx}_s,$ we have: $\forall \vu,$
\begin{align}
&\; \sum_{s=1}^S K a_s \widehat{\Gap}(\vxt_s; \vu) =  \sum_{s=1}^S \sum_{k=1}^K  a_s\innp{\mF(\vu) + g'(\vu), \vx_{s,k} - \vu}    \nonumber\\
% =&\; \sum_{k=1}^K  a_s\innp{\mF(\vu) + g'(\vu), \vxt_{s} - \vu}    \nonumber\\
\leq&\;\sum_{s=1}^S \sum_{k=1}^K a_s \big(\innp{\mF(\vx_{s,k}), \vx_{s,k} - \vu} + g(\vx_{s,k}) - g(\vu) - \frac{\gamma}{2}\|\vx_{s,k} - \vu\|^2\big).\label{eq:gap-vr-1}  
% =&\; \sum_{s=1}^S \sum_{k=1}^K a_s\Big(\overline{\Gap}(\vx_{s,k}; \vu) - \frac{\gamma}{2}\|   \vx_{s,k} - \vu\|^2\Big). 
\end{align}
Meanwhile, in Eq.~\eqref{eq:gap-vr-1},  we also have 
\begin{align}
 \sum_{s=1}^S\sum_{k=1}^K & a_s \big(\innp{\mF(\vx_{s,k}), \vx_{s,k} - \vu} + g(\vx_{s,k}) - g(\vu)\big)  \nonumber\\
    =\; & \sum_{s=1}^S\sum_{k=1}^K a_s \big(\innp{\vq_{s,k}, \vx_{s,k} - \vu} + g(\vx_{s,k}) - g(\vu)\big) -  \psi_{1,0}(\vu) \nonumber\\
    &+ \psi_{1,0}(\vu) + \sum_{s=1}^S \sum_{k=1}^K a_s \innp{\mF(\vx_{s,k}) - \vq_{s,k}, \vx_{s,k} - \vu} \nonumber\\
    =\; &-\psi_{S,K}(\vu) +  \psi_{1,0}(\vu) + \sum_{s=1}^S\sum_{k=1}^K a_s\innp{\mF(\vx_{s,k}) - \vq_{s,k}, \vx_{s,k} - \vu}\nonumber\\
     \le\; &  - \Big(\psi_{S,K}(\vx_{S,K}) + \frac{K(1 + A_S \gamma)}{2}\|\vu - \vx_{S,K}\|^2\Big)  
    +\psi_{1,0}(\vu) \nonumber\\
    \;& + \sum_{s=1}^S\sum_{k=1}^K a_s \innp{\mF(\vx_{s,k}) - \vq_{s,k}, \vx_{s,k} - \vu}, \label{eq:gap-vr-2}
\end{align}
where the third equality is by the definition of $\psi_{s,K}(\vu)$, and the last inequality is by the optimality condition of $\vx_{S,K}$ and $K(1 + A_S\gamma)$-strong convexity of $\psi_{S,K}$, which leads to $\psi_{S,K}(\vu) \ge \psi_{S,K}(\vx_{S,K}) + \frac{K(1 +  A_S\gamma)}{2}\|\vu - \vx_{S,K}\|^2.$ 

% On the other hand, for $s\ge 1, k\in[K]$, $$\psi_{s, k-1}(\vx_k) - \psi_{s, k-1}(\vx_{k-1}) \geq \frac{K(1 + A_{s-1}\gamma)}{2}\|\vx_{s, k} - \vx_{s, k-1}\|^2$$ (as discussed in the paragraph above). Meanwhile, by definition, we also have $\psi_{s, K} = \psi_{s+1, 0}$, $\vx_{s, K} = \vx_{s+1,0}$, $\psi_{1,0}(\vx_{1,0}) = \psi_{1,0}(\vx_{0}) = 0$. 

On the other hand, we also have
\begin{align}
 \psi_{S,K}(\vx_{S,K}) =\;&\sum_{s=1}^S \sum_{k=1}^K (\psi_{s,k}(\vx_{s,k}) - \psi_{s,k-1}(\vx_{s,k-1})) + \psi_{1,0}(\vx_{1,0}) \nonumber\\
 = \;&\sum_{s=1}^S \sum_{k=1}^K (\psi_{s,k-1}(\vx_{s,k}) - \psi_{s,k-1}(\vx_{s,k-1}))\nonumber\\
 \ge\;& \sum_{s=1}^S\sum_{k=1}^K \frac{K(1+ A_{s-1}\gamma)}{2}\|\vx_{s,k} - \vx_{s,k-1}\|^2, \label{eq:gap-vr-3}
\end{align}
where the first equality is by $\psi_{s, K} = \psi_{s+1, 0}$, $\vx_{s, K} = \vx_{s+1,0}$, 
the second equality is by the fact that $\psi_{s,k}(\vx_{s,k})=\psi_{s,k-1}(\vx_{s,k})$ and $\psi_{1,0}(\vx_{1,0}) = \psi_{1,0}(\vx_{0}) = 0$, the last inequality is by the $K(1+A_{s-1}\gamma)$-strong convexity of $\psi_{s, k-1}(\cdot)$ for $k\ge 1$ and the optimality of $\vx_{s,k-1}.$ 

Combining Eqs.~\eqref{eq:gap-vr-1}--\eqref{eq:gap-vr-3}, using the definition of $E_{s,k}(\vu)$ and that, by definition, $\vx_{s+1,0}=\vx_{s,K}$, the claimed bound follows by dividing both sides by $K$. %and dividing by $K$ on both sides, we obtain Lemma \ref{lem:gap-vr}. 
\end{proof}
%

% By Lemma \ref{lem:gap-vr}, we 

%In the following, for $s\ge 1, k\in[K]$, let $\gF_{s,k}$ denote the natural filtration, containing all the randomness in the algorithm up to and including the $k$-th iteration of $s$-th epoch. For the error terms $\{E_{s,k}(\vu)\}_{k\geq 1}$ in Lemma \ref{lem:gap-vr}, by the definition of the variance reduced extrapolated operator $\vq_{s,k}, $ we have Lemma \ref{lem:error-vr}.
The following lemma bounds the cumulative error terms in expectation and its proof is provided in Appendix~\ref{appx:proof-lemma-error-vr}. 
\begin{lemma}\label{lem:error-vr}
For any fixed $\vu\in\dom(g)$, taking expectation on all the randomness in the history, we have: {in Algorithm~\ref{alg:coder-vr} and under Assumptions~\ref{ass:Lip} and \ref{assmpt:new-Lip} applying to each $F_t$,} for all $S\ge 1, K\ge 1,$
\begin{align}
\sum_{s=1}^S\sum_{k=1}^K \E\Big[E_{s,k}(\vu) \Big] 
\le\;&\sum_{j=1}^m  a_S\E[\langle \mF^j(\vx_{S+1,0}) - \mF^j(\vy_{S+1, 0, j}),  \vx_{S+1,0}^j - \vu^j\rangle]  \nonumber   \\ 
&\; -  \frac{K(1+A_{S-1}\gamma)}{8} \E[ \| {\vx}_{S+1,0} - \vx_{S+1,-1}\|^2]  \nonumber   \\
&\; - \frac{\beta K a_{S+1} }{2}\E[\|\hat{\vx}_S- \vu\|^2] + \frac{\beta K a_{1}}{2}\|{\vx}_{0} - \vu\|^2,
\end{align}
where $\vxh_S = \frac{1}{K}\sum_{k=1}^K\big( \frac{\beta }{\beta+\gamma}\vx_{S, k-1} + \frac{\gamma}{\beta+\gamma}\vx_{S,k}\big), \vx_{S+1, 0} = \vx_{S, K}, \vx_{S+1, -1} = \vx_{S, K-1}$.
\end{lemma}

We are now ready to state and prove the main result of this section.
\begin{theorem}\label{thm:main-coder-vr}
Let $\vx_0 \in \dom(g)$ be an arbitrary initial point and $\{\vx_{s,k}\}_{s\ge 1, k\in[K]}$ evolve according to Algorithm~\ref{alg:coder-vr}. Then under Assumptions~\ref{assmpt:general}, \ref{ass:mono}, and~\ref{assmpt:strongly-convex}, and Assumptions~\ref{ass:Lip} and \ref{assmpt:new-Lip} applying to all $\mF_t$, $t \in \{1, \dots, n\}$, we have that  
for all $S\ge 1, \vu \in \dom(g):$ 
\begin{align}
A_S \E\Big[\widehat{\Gap}\Big(\frac{1}{A_S}\sum_{s=1}^{S} a_s\vxt_s; \vu\Big)\Big] +  \frac{1 +  A_S \gamma}{4}\E\big[ \| \vx_{S+1,0} - \vu\|^2] \le   \frac{5}{8}\|\vu - \vx_{0} \|^2, 
\end{align}
where the expectation is w.r.t.~all the randomness in the algorithm. 
In particular, 
\begin{equation}\notag
\E\Big[\widehat{\Gap}\Big(\frac{1}{A_S}\sum_{s=1}^{S} a_s\vxt_s; \vu\Big) \Big] \le\frac{5}{8 A_S}\|\vu - \vx_0\|^2.  
\end{equation}
Further, if $\vx^*$ is any solution to Problem \eqref{eq:main-problem}, we also have %\jd{wouldn't strong convexity of $g$ guarantee uniqueness of $\vx^*$ in this case?}
\begin{equation}\notag
    \E\big[\|\vx_{S+1, 0} - \vx^*\|^2\big] \leq \frac{5}{2(1 + A_S \gamma)}\|\vx_0 - \vx^*\|^2.
\end{equation}
Let $\tau = \min\Big\{ \frac{\sqrt{K}}{8L},  \frac{K}{8\hat{L}}\Big\}.$ Then $A_S$ grows at least as fast as
$$A_S \ge %\min\Big\{ \frac{2L\tau}{\sqrt{K}\gamma}\Big(\Big( 1+\frac{\sqrt{K}\gamma}{2L} \Big)^S -1\Big)\Big),
\tau\,{\max}\{S,(1+\gamma\tau)^{S-1}\}. 
%\tau(1+\tau\gamma)^{S-1}   \Big\}\Big\}.
$$  %
\end{theorem}
\begin{proof}
Combining Lemmas \ref{lem:gap-vr} and \ref{lem:error-vr}, %
\begin{equation}
\begin{aligned}
&\; A_S \E\Big[\widehat{\Gap}\big( \frac{1}{A_S}\sum_{s=1}^{S} a_s\vxt_s; \vu)\Big] =\sum_{s=1}^S a_s \E\big[\widehat{\Gap}(\vxt_s; \vu) \big]   \\  
\le&\;    \frac{1}{2}\|\vu - \vx_{0}\|^2 -  \E\Big[\frac{1 + A_S \gamma}{2}\|\vu - \vx_{S+1,0}\|^2 \Big]  \\ 
&\; + \frac{1}{K} \sum_{j=1}^m a_S\E[\langle \mF^j(\vx_{S+1,0}) - \mF^j(\vy_{S+1, 0, j}),  \vx_{S+1,0}^j - \vu^j\rangle]     \\ 
&\; -  \frac{1+A_{S-1}\gamma}{8} \E[ \| {\vx}_{S+1,0} - \vx_{S+1,-1}\|^2 ]     \\
&\; - \frac{\beta a_{S+1} }{2}\E[\|\hat{\vx}_S- \vu\|^2] + \frac{\beta a_{1}}{2}\|\vx_{0} - \vu\|^2. 
\end{aligned}\label{eq:main-coder-vr-1}
\end{equation}

To prove the theorem, we only need to bound the inner product terms in Eq.~\eqref{eq:main-coder-vr-1}, and then use our parameter choices to cancel out any terms that do not come from the initialization. 

%Then in Eq.~\eqref{eq:main-coder-vr-1}, 
Using Cauchy-Schwarz inequality and Young's inequality, we have,
for all $j\in[m]$, 
\begin{align}
\;&\frac{a_S}{K}\langle \mF^j(\vx_{S+1,0}) - \mF^j(\vy_{S+1, 0, j}),  \vx_{S+1,0}^j - \vu^j\rangle  \nonumber  \\
%\le\;&\frac{a_S}{K}\|\mF^j(\vx_{S+1,0}) - \mF^j(\vy_{S+1, 0, j})\| \|\vx_{S+1,0}^j - \vu^j\| \nonumber  \\    
\le\;&  \frac{a_S^2}{K^2(1+A_S\gamma)}\|\mF^j(\vx_{S+1,0}) - \mF^j(\vy_{S+1, 0, j})\|^2 + \frac{1+A_S\gamma}{4} \|\vx_{S+1,0}^j - \vu^j\|^2.  \label{eq:main-coder-vr-2}
\end{align}
%In Eq.~\eqref{eq:main-coder-vr-2}, with 
Further, by Assumption \ref{assmpt:new-Lip} and the definition of $\mQh^j$ in Eq.~\eqref{eq:Q-hat}, % we have 
\begin{align}
\|\mF^j(\vx_{S+1,0}) - \mF^j(\vy_{S+1, 0, j})\|^2 \le\;& (\vx_{S+1,0}-\vy_{S+1, 0, j})^T\mQ^j(\vx_{S+1,0}-\vy_{S+1, 0, j})     \nonumber  \\ 
=\;& (\vx_{S+1,0}-\vx_{S+1, -1})^T\mQh^j(\vx_{S+1,0}-\vx_{S+1, -1}). \notag %\label{eq:main-coder-vr-3}
\end{align}

Hence: 
\begin{align}
\sum_{j=1}^m \|\mF^j(\vx_{S+1,0}) -& \mF^j(\vy_{S+1, 0, j})\|^2  \nonumber  \\ 
%\le\;& \sum_{j=1}^m (\vx_{S+1,0}-\vx_{S+1, -1})^T\mQh^j(\vx_{S+1,0}-\vx_{S+1, -1})  \nonumber  \\ 
 \le\;&  (\vx_{S+1,0}-\vx_{S+1, -1})^T\Big(\sum_{j=1}^m\mQh^j\Big)(\vx_{S+1,0}-\vx_{S+1, -1})  \nonumber  \\ 
 \le\;& \hat{L}^2\|\vx_{S+1,0}-\vx_{S+1, -1}\|^2.  \label{eq:main-coder-vr-3}
\end{align}

Thus, combining Eqs.~\eqref{eq:main-coder-vr-2} and \eqref{eq:main-coder-vr-3}, with the definition of $a_S$, we have 
\begin{align}
\;&\frac{a_S}{K} \sum_{j=1}^m\langle  \mF^j(\vx_{S+1,0}) - \mF^j(\vy_{S+1, 0, j}),  \vx_{S+1,0}^j - \vu^j\rangle  \nonumber  \\ 
\le\;& \frac{1 + A_{S-1}\gamma}{64}\|\vx_{S+1,0}-\vx_{S+1, -1}\|^2 + \frac{1+A_S\gamma}{4} \|\vx_{S+1,0} - \vu\|^2. \label{eq:main-coder-vr-4}
\end{align}

To complete bounding $\widehat{\Gap}$, it remains to combine Eqs.~\eqref{eq:main-coder-vr-1} and \eqref{eq:main-coder-vr-4} and use that, by our setting, $A_0 = 0$ and $a_1 \beta \le \frac{1}{4}.$

It remains to bound below the growth of $\{A_s\}_{s\ge 1}.$ Let $\tau = \min\big\{\frac{\sqrt{K}}{8L}, \frac{K}{8 \hat{L}}\big\}$. When $\gamma = 0,$  using the definition of step size from Step~15 in Algorithm~\ref{alg:coder-vr},  it is easy to verify that $a_s = \tau$ and $A_s = \tau s.$ When $\gamma > 0,$ 
as $\frac{1}{\beta}\ge \tau, $ we have 
\begin{align}
a_{s+1} \geq\;& \min\big\{(1 +\gamma\tau)a_s, (1 + \gamma A_s)\tau\big\} \ge     \min\big\{(1 +\gamma\tau)a_s,  \gamma\tau A_s\big\}. \label{eq:main-coder-vr-5}
\end{align}
Then we use mathematical induction to prove that for all $s\ge 1,$ $a_s \ge \gamma\tau A_{s-1}.$ 

First, for $s=1,$ we know that $a_1 = \tau \ge 0 = A_0.$ Second, assume that for an $s\ge 1,$ $a_s \ge \gamma\tau A_{s-1}.$ Then as $(1 +\gamma\tau)a_s\ge  \gamma\tau A_{s-1} + \gamma\tau a_s  =  \gamma\tau A_s$, by Eq.~\eqref{eq:main-coder-vr-5}, we have $a_{s+1}\ge \gamma\tau A_s.$ As a result, we have $A_{s+1}\ge (1+\gamma\tau) A_s  \ge (1+\gamma\tau)^{s}A_1 = \tau(1+\gamma\tau)^{s}.$
Hence, we can conclude that for all $s\ge 1,$ $A_s = \tau\,{\max}\{s,(1+\gamma\tau)^{s-1}\}.$ 
\end{proof}

In Theorem \ref{thm:main-coder-vr}, the inner number of iterations $K$ can be set to any positive integer. So, in Algorithm \ref{alg:coder-vr}, to balance the computational cost between outer loop and inner loop, we can set $K = \Theta(n)$ {(any constant times $n$; for example, we can set $K = n$).} 
%\markupdelete{Let $\tau' = \min\big\{\frac{\sqrt{n}}{L}, \frac{n}{\hat{L}}\big\}$. Then by Theorem \ref{thm:main-coder-vr}, in the general convex case (i.e., $\gamma=0$), to obtain an $\epsilon$-accurate solution $\vx^{\mathrm{out}} =  \frac{1}{A_S}\sum_{s=1}^{S} a_s\vxt_s$  such that $\sup_{\vu\in\gB}\E\Big[\widehat{\Gap}(\vx^{\mathrm{out}}; \vu)\Big]\le \epsilon$ with $\gB\in\dom(g)$ of radius $D$, we need at most $S = O(\frac{D^2}{\tau'\epsilon})$  epochs and  $O(\frac{nD^2}{\tau'\epsilon})$   arithmetic operations;  in the strongly convex case (i.e., $\gamma>0$), to obtain an $\epsilon$-accurate solution $\vx^{\mathrm{out}} =  \vx_{S+1,0}$ such that $ \E[\|\vx^{\mathrm{out}} - \vx^*\|^2] \le  \epsilon,$ we need at most  $O\Big(\frac{\log\frac{\|\vx_0-\vx^*\|^2}{\gamma\tau' \epsilon}}{ \log( 1 + \gamma\tau')}\Big)$ epochs and $O\Big(\frac{n \log\frac{\|\vx_0-\vx^*\|^2}{\gamma\tau' \epsilon}}{ \log( 1 + \gamma\tau')}\Big)$ arithmetic operations.} 
{The choice of $K$ affects the parameter $\tau = \min\{\frac{\sqrt{K}}{L}, \frac{K}{\hat{L}}\}.$ Based on the bound on the gap in Theorem 4.1, to have gap at most $\epsilon,$ it suffices that $A_S \geq \frac{5D^2}{8\epsilon}$. When $\gamma = 0$ (general monotone case), $A_S \geq \tau S,$ so it suffices that $S \geq \frac{5D^2}{8\epsilon\tau}.$ When $\gamma > 0,$ $A_S \geq \tau(1 + \gamma \tau)^{S-1}$, so it suffices that $S \geq 1 + \frac{\log(\frac{5D^2}{8\epsilon\tau})}{\log(1 + \gamma\tau)}.$ One full epoch requires $K$ iterations, each of which does one full cycle over the coordinates. Thus, the total number of arithmetic operations, assuming the cost of evaluating a block of size $d/m$ of $F$ is $c d/m$, is $SKcd = O(cnd\min\{\frac{D^2}{\epsilon\tau} ,\, \frac{\log(\frac{D^2}{\epsilon\tau})}{\log(1+\gamma \tau)}\})$.}  

{\begin{remark}
    It is worth noting that unlike CODER, VR-CODER has a complexity guarantee that depends both on the traditional Lipschitz constant $L$ and our newly introduced constant $\hat{L}.$ This is a consequence of our analysis: there is exactly one term in the analysis that requires Lipschitz constant $L$ (see the last inequality in \eqref{eq:error-vr-4}). It is unclear whether the dependence on $L$ can be avoided; from our current analysis, this does not seem possible. 
\end{remark}}

\section{{Numerical Experiment}}\label{sec:num-exp}

{To illustrate the performance of proposed algorithms, we conducted a preliminary numerical experiment on convex $\ell_1$ norm-regularized and convex elastic net-regularized SVM problems, reformulated as GMVI problems~\eqref{eq:main-problem}, as described in Example~\ref{ex:elastic-net} and Example~\ref{exam:svm}. In the experiment, we compared CODER and VR-CODER against Randomized Accelerated Primal-Dual (RAPD) algorithm~\cite{hamedani2018iteration}, PCCM, and PRCM\footnote{{PCCM and PRCM are described in Remark~\ref{rem:non-convergence-of-vanilla-methods}.}}. RAPD was chosen for comparison as the most closely related method to CODER: it performs similar gradient extrapolation, but takes randomized coordinate updates on the primal side and full vector updates on the dual side. PRCM and PCCM were selected for comparison to illustrate that the extrapolation step used by CODER does not harm the convergence speed. As discussed in Remark~\ref{rem:non-convergence-of-vanilla-methods}, the extrapolation step is necessary under arbitrary block separation of the coordinates, as without it, the algorithm would diverge in general (and, in particular, PCCM and PRCM are divergent under general block separation, but converge in our experiments, as we use single coordinate blocks). }

\begin{figure*}[t!]
    \centering
    \includegraphics[width=0.49\textwidth]{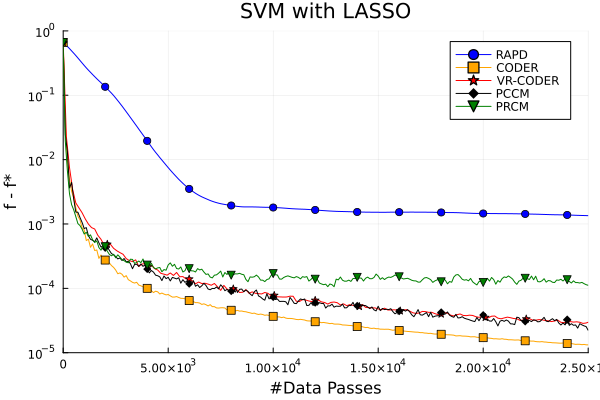} \label{fig:a1a-svm+lasso}%
    \hspace{\fill}
    \includegraphics[width=0.49\textwidth]{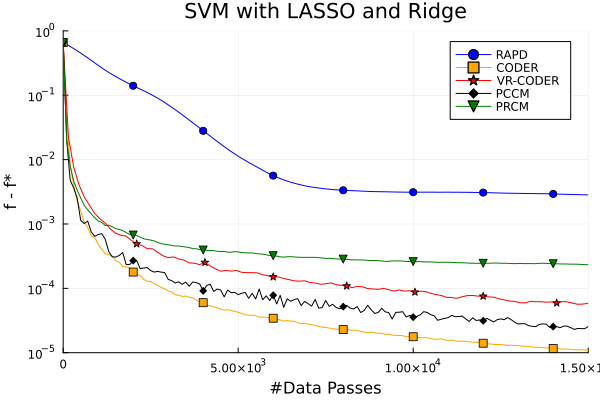} \label{fig:a1a-svm+lasso+ridge}
    \caption{{Comparison of implemented algorithms in terms of the number of full data passes. Left: SVM with LASSO regularization parameter set to $10^{-4}.$ Right: SVM with the elastic net regularization, where both LASSO and ridge regularization parameters are set to $10^{-4}.$}}%
\end{figure*}

{Fig.~\ref{fig:a1a-svm+lasso} shows the performance of the implemented algorithms in terms of the optimality gap for the original SVM problem (see Example~\ref{exam:svm} and the discussion from the introduction) on LibSVM a1a dataset, which is a sparse dataset with dimensions $d = 123$, $n = 1605$~\cite{chang2011libsvm}. In both plots, the $x$-axis displays the computational cost measured by the number of full gradient evaluations, which require one full pass over the data set. Note that since we are using SVRG-type variance reduction for VR-CODER, there is an additional full gradient computation (Step 4 in Algorithm~\ref{alg:coder-vr}) for each outer loop, hence one epoch corresponds to two data passes when $K = n$. For each algorithm, we tune each of the algorithm parameters to the best of our ability and display the best performing results. Our code is available at \url{https://github.com/ericlincc/CODER}.}

{As can be observed from the figures, CODER displays the fastest convergence compared to other algorithms, with only PCCM being competitive with it, while the implemented randomized algorithms RAPD and PRCM initially converge fast but then flatten out and progress according to a much slower convergence rate. Interestingly, VR-CODER on this example does not show the theoretical improvement compared to CODER. We conjecture that this is due to the data set not being large enough to offset the complexity of the variance reduction scheme. Nevertheless, VR-CODER remains competitive. We leave further investigation of empirical performance of CODER and VR-CODER for future work.}

\section{Conclusion}

We presented novel extrapolated cyclic coordinate method CODER and its variance-reduced counterpart for the finite-sum setting---VR-CODER, which provably converge on the class of generalized variational inequalities. This class includes convex composite optimization and convex-concave min-max optimization. CODER and VR-CODER are the first cyclic coordinate methods that provably converge on this broad class of problems. Further, for the special case of composite convex optimization problems, CODER provides improved convergence guarantees in terms of the dependence on the number of coordinate blocks compared to the state of the art. These results are enabled by a novel Lipschitz condition for the gradients that we introduced.  
Some open questions that merit further investigation remain. For example, one such question is understanding the complexity of standard optimization problem classes under our new Lipschitz condition by obtaining new oracle lower bounds.  

\section*{{Acknowledgements}}

{We are indebted to Cheuk Yin (Eric) Lin, who fully implemented the algorithms from Section~\ref{sec:num-exp}. }

\bibliographystyle{siamplain}
\bibliography{main.bib}

\appendix

\section{(Lipschitz) Parameter-Free CODER}\label{appx:param-free}
\begin{algorithm}[htb!]
\caption{Cyclic cOordinate Dual avEraging with extRapolation (CODER)}\label{alg:coder-PF}
\begin{algorithmic}[1]
\STATE \textbf{Input:} $\vx_0\in\mathrm{dom}(g), \gamma \geq 0, \hat{L}_0 > 0, m, \{\gS^1, \dots, \gS^m\}$
\STATE \textbf{Initialization:} $\vx_{-1} = \vx_{0}, \vp_0 = \mF(\vx_0)$, $a_0= A_0 = 0$.
\STATE $\psi_{0}^i(\vx^i) = \frac{1}{2}\|\vx^i - \vx_0^i\|^2,$ $1\leq i\leq m$. 
\FOR{$k = 1$ to $K$} 
\STATE $\hat{L}_k = \hat{L}_{k-1}/2$.
\REPEAT
\STATE $\hat{L}_k = 2 \hat{L}_k$.
\STATE $a_k =  \frac{1+\gamma A_{k-1}}{2{\hat{L}_k}}, A_k = A_{k-1} + a_k.$ 
\FOR{$j = 1$ to $m$} %
\STATE $\vp_k^j =   \mF^{j}(\vx^{1}_{k}, \ldots, \vx^{j-1}_{k}, \vx^{j}_{k-1},\ldots   \vx^{m}_{k-1})$
\STATE $\vq^j_k = \vp_k^j + \frac{a_{k-1}}{a_k}(\mF^j(\vx_{k-1}) - \vp_{k-1}^j)$
\STATE $\vz_k^j = \vz_{k-1}^j + a_k\vq^j_k$
\STATE $\vx_{k}^j = \mathrm{prox}_{A_k g^j}(\vx_0 - \vz_k^j)   $
\ENDFOR
\UNTIL{$\|\mF(\vx_k) - \vp_k\| \leq \hat{L}_k \|\vx_{k} - \vx_{k-1}\|$}
\ENDFOR
\STATE \textbf{return} $\vx_K$, $\tilde{\vx}_K = \frac{1}{A_K}\sum_{k=1}^{K}a_k\vx_k.$
\end{algorithmic}	
\end{algorithm}

CODER, as stated in Algorithm~\ref{alg:coder}, requires knowledge of the Lipschitz parameter $\hat{L}.$ This may seem like a limitation of our approach, especially since the Lipschitzness of $\mF$ assumed in our work is much different from the traditional Lipschitz assumptions for either the full gradient or its (block) coordinate components.

It turns out that the explicit knowledge of $\hat{L}$ is not required  at all for our algorithm to work, at least whenever the permutation over the blocks is fixed throughout the algorithm execution. This is revealed by our analysis, as the only place in the analysis where the Lipschitz assumption on $\mF$ is used is to verify that $\|\mF(\vx_k) - \vp_k\| \leq \hat{L}\|\vx_k - \vx_{k-1}\|$. By the argument used in the proof of Theorem~\ref{thm:main-coder} and the Lipschitz assumption on $\mF$ (Assumption~\ref{assmpt:general}), this condition must be satisfied for any $\hat{L} \geq \big\|\sum_{j}\mQh^j\big\|.$ Thus, a natural approach is to start with some initial estimate $\hat{L}_0 > 0$ of $\hat{L}$ and double it each time the condition $\|\mF(\vx_k) - \vp_k\| \leq \hat{L}\|\vx_k - \vx_{k-1}\|$ fails. The total number of times that this estimate can be doubled is then bounded by $\log_2(\frac{2\hat{L}}{\hat{L}_0}),$ and, under a mild assumption that $\hat{L}_0 = O(\hat{L})$ and $\hat{L}_0$ is not overwhelmingly (e.g., exponentially in $1/\epsilon, n$) smaller than $\hat{L}$, the total overhead due to estimating $\hat{L}$ is absorbed by the convergence bound from Theorem~\ref{thm:main-coder}. 
The variant of CODER that implements this doubling trick is summarized in Algorithm~\ref{alg:coder-PF}. 

\section{Proof of Lemma \ref{lem:error-vr}}\label{appx:proof-lemma-error-vr}
To prove Lemma~\ref{lem:error-vr}, we first provide an upper bound for the operator extrapolation error $a_s\langle\mF^j(\vx_{s,k}) - \vq_{s,k}^j,  \vx_{s,k}^j - \vu^j\rangle$ in Lemma~\ref{lemma:coder-gap-change-vr}, by using the definition of $\vq_{s,k}^j$. We then prove Lemma~\ref{lem:error-vr} by bounding the individual terms from this upper bound on the error.  

\begin{lemma}\label{lemma:coder-gap-change-vr}
For all $s\ge 1$, $k\in[K],j\in [m]\!$ and $\vu\in\dom(g)$, the sequence $\{\vx_{s,k}\}$ of Algorithm \ref{alg:coder-vr} satisfies 
\begin{equation*}
\begin{aligned}
 &\; a_s\langle\mF^j(\vx_{s,k}) - \vq_{s,k}^j,  \vx_{s,k}^j - \vu^j\rangle     \\
=&\; a_s\langle \mF^j(\vx_{s,k}) - \mF^j(\vy_{s, k, j}),  \vx_{s,k}^j - \vu^j\rangle       \\
&\; -a_{s,k-1}\langle\mF_t^j(\vx_{s, k-1}) - \mF^j_t(\vy_{s,k-1,j}),  \vx_{s,k}^j - \vx_{s,k-1}^j\rangle   \\
&\; + a_s\langle \mF^j(\vy_{s, k, j}) - \big(  \mF^j_t(\vy_{s,k,j}) - \mF_t^{j}(\hat{\vx}_{s-1}) + \boldsymbol{\mu}_{s}^j\big), \vx_{s,k}^j -  \vx_{s,k-1}^j\rangle   \\
&\;+a_s\langle  \mF^j(\vy_{s, k, j}) - \big(  \mF^j_t(\vy_{s,k,j}) - \mF_t^{j}(\hat{\vx}_{s-1}) + \boldsymbol{\mu}_{s}^j\big) ,   \vx_{s,k-1}^j - \vu^j\rangle \\
&\;- a_{s,k-1}\langle\mF_t^j(\vx_{s, k-1}) - \mF_t^j(\vy_{s, k-1, j}),  \vx_{s,k-1}^j - \vu^j\rangle  \\
&\;-\frac{ \beta a_s}{2}( \|\vx_{s,k}^j -  \hat{\vx}_{s-1}^j\|^2 - \|\vx^j_{s, k-1} - \vx_{s,k}^j \|^2\\
&\;\quad\quad\quad\quad + \| \vx^j_{s, k-1} - \vu^j\|^2 - \|\hat{\vx}_{s-1}^j -  \vu^j\|^2) \\
% &\;-a_{s,k-1}\langle\big(\mF_t^j(\vx_{s, k-1}) - \mF_t^j(\vy_{s,k,j})\big) - \big(\mF^j(\vx_{s, k-1}) - \mF^j(\vy_{s,k,j})\big),  \vx_{s,k-1}^j - \vu^j\rangle.
\end{aligned} 
\end{equation*}
\end{lemma}

\begin{proof}
% That is, $\vy_{s, k, j}$ is defined so that $\vp^j_{s,k} = \mF_t^j(\vy_{s, k, j}).$  \cb{Each $\vx_{s,k}^1$ here is a column vector, so the above definition may be not correct.}
% \cb{$\vy_{s,k,j}$ should be defined for the lemma, not only in the proof.} \cb{The definition of $\vy_{s,k,j}$ should be given outside of the lemma. }

%  $\mF^j_t(\vy_{s,k,j}) =   \mF_t^{j}(\vx^{1}_{s,k}, \ldots, \vx^{j-1}_{s,k}, \vx^{j}_{s, k-1},\ldots   \vx^{m}_{s, k-1}) $
 
% $\vq^j_{s,k} = \mF^j_t(\vy_{s,k,j}) - \mF_t^{j}(\hat{\vx}_{s-1}) + \boldsymbol{\mu}_{s} + \frac{a_{s-1}}{a_s}(\mF_t^j(\vx_{s, k-1}) - \mF^j_t(\vy_{s,k-1,j})) + \alpha(\vx^j_{s, k-1} - \hat{\vx}_{s-1}^j)$

By the definition of $\vq_{s,k}^j$, we have
\begin{align}
 &\; a_s (\mF^j(\vx_{s,k}) - \vq_{s,k}^j)  \notag\\
 =&\; a_s (\mF^j(\vx_{s,k}) - \mF^j(\vy_{s, k, j})) + a_s(\mF^j(\vy_{s, k, j}) - \vq_{s,k}^j) \notag\\
 =&\;  a_s (\mF^j(\vx_{s,k}) - \mF^j(\vy_{s, k, j})) + a_s\big(\mF^j(\vy_{s, k, j}) - \big(  \mF^j_t(\vy_{s,k,j}) - \mF_t^{j}(\hat{\vx}_{s-1}) + \boldsymbol{\mu}_{s}^j\big)\big)   \notag\\
 &\;- a_{s,k-1} (\mF_t^j(\vx_{s, k-1}) - \mF^j_t(\vy_{s,k-1,j})) - \beta a_s (\vx^j_{s, k-1} - \hat{\vx}_{s-1}^j). \label{eq:coder-gap-change-vr-1}
\end{align}
To prove the lemma, it remains to take the inner product between the right-hand side of Eq.~\eqref{eq:coder-gap-change-vr-1} and bound the corresponding terms. 
%Using Eq.~\eqref{eq:coder-gap-change-vr-1}, we can bound $a_s\langle \mF^j(\vx_{s,k}) - \vq_{s,k}^j,  \vx_{s,k}^j - \vu^j\rangle$ by bounding the last three terms w.r.t.~the   right hand side of Eq.~\eqref{eq:coder-gap-change-vr-1}. 
% \cb{Introduce the filtration notation. $\gF_{s,k, j-1}$, $\gF_{s,K} = \gF_{s+1,0}$  }

First,  we have 
\begin{align}
&\;a_s\langle \mF^j(\vy_{s, k, j}) - \big(  \mF^j_t(\vy_{s,k,j}) - \mF_t^{j}(\hat{\vx}_{s-1}) + \boldsymbol{\mu}_{s}^j\big), \vx_{s,k}^j - \vu^j\rangle        \notag\\
=&\;a_s\langle \mF^j(\vy_{s, k, j}) - \big(  \mF^j_t(\vy_{s,k,j}) - \mF_t^{j}(\hat{\vx}_{s-1}) + \boldsymbol{\mu}_{s}^j\big), \vx_{s,k}^j -  \vx_{s,k-1}^j\rangle   \notag\\
+&\;a_s\langle  \mF^j(\vy_{s, k, j}) - \big(  \mF^j_t(\vy_{s,k,j}) - \mF_t^{j}(\hat{\vx}_{s-1}) + \boldsymbol{\mu}_{s}^j\big) ,   \vx_{s,k-1}^j - \vu^j\rangle.     \label{eq:coder-gap-change-vr-2}
% =&\;a_s\langle \mF^j(\vy_{s, k, j}) - \big(  \mF^j_t(\vy_{s,k,j}) - \mF_t^{j}(\hat{\vx}_{s-1}) + \boldsymbol{\mu}_{s}^j\big), \vx_{s,k}^j -  \vx_{s,k-1}^j\rangle ,
\end{align}
% where the last equality is by the fact that $\mF^j_t(\vy_{s,k,j}) = \mF^j(\vy_{s, k, j})$, $\mF_t^{j}(\hat{\vx}_{s-1}) = \mF^{j}(\hat{\vx}_{s-1})=\boldsymbol{\mu}_{s}^j$. 
% and thus $\mF^j(\vy_{s, k, j}) - \big(  \mF^j_t(\vy_{s,k,j}) - \mF_t^{j}(\hat{\vx}_{s-1}) + \boldsymbol{\mu}_{s}\big)  = \vzero.$

Second, 
\begin{align}
&\;a_{s,k-1} \langle\mF_t^j(\vx_{s, k-1}) - \mF^j_t(\vy_{s,k-1,j}), \vx_{s,k}^j - \vu^j\rangle         \notag\\
=&\;a_{s,k-1} \langle\mF_t^j(\vx_{s, k-1}) - \mF^j_t(\vy_{s,k-1,j}), \vx_{s,k}^j - \vx_{s,k-1}^j\rangle  \notag\\
&\; + a_{s,k-1} \langle \mF_t^j(\vx_{s, k-1}) - \mF^j_t(\vy_{s,k-1,j}),  \vx_{s,k-1}^j - \vu^j\rangle. 
% =&\; a_{s,k-1} \langle\mF_t^j(\vx_{s, k-1}) - \mF^j_t(\vy_{s,k-1,j}), \vx_{s,k}^j - \vx_{s,k-1}^j\rangle  \notag\\
% &\; + a_{s,k-1}\langle\mF^j(\vx_{s, k-1}) - \mF^j(\vy_{s,k,j}),  \vx_{s,k-1}^j - \vu^j\rangle       \notag\\
% &\; + a_{s,k-1}\langle\big(\mF_t^j(\vx_{s, k-1}) - \mF_t^j(\vy_{s,k,j})\big) - \big(\mF^j(\vx_{s, k-1}) - \mF^j(\vy_{s,k,j})\big),  \vx_{s,k-1}^j - \vu^j\rangle. \notag
\label{eq:coder-gap-change-vr-3}
\end{align}

Third, we have the following identity,
\begin{align}
- \beta a_s \langle \vx^j_{s, k-1} - \hat{\vx}_{s-1}^j,   \vx_{s,k}^j - \vu^j\rangle    
 =&\;- a_s \beta \langle \vx^j_{s, k-1} - \hat{\vx}_{s-1}^j,  \vx_{s,k}^j -  \hat{\vx}_{s-1}^j\rangle\notag\\
 &\;- a_s \beta \langle \vx^j_{s, k-1} - \hat{\vx}_{s-1}^j,  \hat{\vx}_{s-1}^j -  \vu^j\rangle     \notag\\
% =&\;-\frac{a_s \beta}{2}\big(\|\vx^j_{s, k-1} - \hat{\vx}_{s-1}^j\|^2 + \|\vx_{s,k}^j -  \hat{\vx}_{s-1}^j\|^2 - \|\vx^j_{s, k-1} - \vx_{s,k}^j \|^2\big) \notag\\
% &\;-\frac{a_s \beta}{2}\big( \| \vx^j_{s, k-1} - \vu^j\|^2 - \|\vx^j_{s, k-1} - \hat{\vx}_{s-1}^j\|^2  - \|\hat{\vx}_{s-1}^j -  \vu^j\|^2 \notag\\
=&\;-\frac{\beta a_s }{2}( \|\vx_{s,k}^j -  \hat{\vx}_{s-1}^j\|^2 - \|\vx^j_{s, k-1} - \vx_{s,k}^j \|^2 \notag\\
 &\; \quad\quad\quad\quad+ \| \vx^j_{s, k-1} - \vu^j\|^2 - \|\hat{\vx}_{s-1}^j -  \vu^j\|^2). \label{eq:coder-gap-change-vr-4}
\end{align}

Combining Eqs. \eqref{eq:coder-gap-change-vr-1}-\eqref{eq:coder-gap-change-vr-4}, with simple rearrangements, completes the proof.  
% \begin{align}
% \end{align}
$ $
\end{proof}

\begin{proof}[Proof of Lemma~\ref{lem:error-vr}]
We prove the lemma by bounding the individual terms from the right-hand side in Lemma~\ref{lemma:coder-gap-change-vr}. We keep the first term from the right-hand side unchanged, and start by bounding the second line term. 
% Without loss of generality, we set $\vx_{s, -1} = \vx_{s, 0}.$ Here we provide upper bounds for the components of the upper bound in Lemma \ref{lemma:coder-gap-change-vr}  further by using Assumption \ref{assmpt:new-Lip} and the variance reduced construction of $\vq_{s,k}^j$.   
%
Using the definition of $\vy_{s, k-1, j}$, Cauchy-Schwarz inequality, and Young's inequality, for all $j$ and all $\alpha_1 > 0,$ we have 
\begin{align}
%&\; - a_{s,k-1} \innp{\mF_t^j(\vx_{s,k-1}) - \vp_{s,k-1}^j, \vx_{s,k}^j - \vx_{s,k-1}^j} \notag\\
%=\; 
& - a_{s,k-1}   \innp{\mF_t^j(\vx_{s,k-1}) - \mF_t^j(\vy_{s,k-1, j}), \vx_{s,k}^j - \vx_{s,k-1}^j}\notag\\
\le\; &  a_{s,k-1}  \|\mF_t^j(\vx_{s,k-1}) - \mF_t^j(\vy_{s,k-1, j})\| \| \vx_{s,k}^j-\vx_{s,k-1}^j\|\notag\\
\leq\;& \frac{{a_{s,k-1}}^2 \alpha_1}{2}   \|\mF_t^j(\vx_{s,k-1}) - \mF_t^j(\vy_{s,k-1, j})\|^2 + \frac{1}{2\alpha_1}\|\vx_{s,k}^j-\vx_{s,k-1}^j\|^2.   \label{eq:error-vr}
\end{align}
Meanwhile, in Eq.~\eqref{eq:error-vr}, by  the definition of $\widehat{\mQ}^j$ in Eq.~\eqref{eq:Q-hat} and Assumption \ref{assmpt:new-Lip}, %we have
\begin{align}
\|\mF_t^j(\vx_{s,k-1}) - \mF_t^j(\vy_{s,k-1, j})\|^2   \le\;&  (\vx_{s,k-1} - \vy_{s,k-1, j})^T \mQ^j (\vx_{s,k-1} - \vy_{s,k-1, j})        \notag \\
=\;&  (\vx_{s,k-1} - \vx_{s,k-2})^T \widehat{\mQ}^j (\vx_{s,k-1} - \vx_{s,k-2}).  \label{eq:error-vr1}
\end{align}

%Hence,
%\begin{equation}\label{eq:error-vr1-sum}
%    \sum_{j=1}^m \|\mF_t^j(\vx_{s,k-1}) - \mF_t^j(\vy_{s,k-1, j})\|^2 \leq \hat{L}^2 \|\vx_{s,k-1} - \vx_{s,k-2}\|^2.
%\end{equation}

For the third line term, for any $\alpha_2>0$, applying Cauchy-Schwarz and Young's inequalities again, 
\begin{align}
& a_s\langle \mF^j(\vy_{s, k, j}) - \big(  \mF^j_t(\vy_{s,k,j}) - \mF_t^{j}(\hat{\vx}_{s-1}) + \boldsymbol{\mu}_{s}^j\big), \vx_{s,k}^j -  \vx_{s,k-1}^j\rangle   \notag\\
\le&\;  \frac{{a_s}^2\alpha_2}{2}\|\mF^j(\vy_{s, k, j}) - \big(  \mF^j_t(\vy_{s,k,j}) - \mF_t^{j}(\hat{\vx}_{s-1}) + \boldsymbol{\mu}_{s}^j\big)\|^2 \!\!+\!\! \frac{1}{2\alpha_2}\|\vx_{s,k}^j -  \vx_{s,k-1}^j\|^2. \label{eq:error-vr-2} 
\end{align}

For all $s\ge 1, k\in[K], j \in[m]$, let $\gF_{s, k, j}$ be the natural filtration with $\gF_{s, k, m} = \gF_{s, k+1, 0},\gF_{s, K, m} = \gF_{s+1, 1, 0}$, containing all randomness up to and including iteration $j$ of cycle $k$ within epoch $s.$ Then in Eq.~\eqref{eq:error-vr-2}, the first variance term can be bounded using standard variance reduction arguments: %\jd{What is the difference between the first two lines?}
\begin{align}
&\;\E\Big[   \|\mF^j(\vy_{s, k, j}) - \big(  \mF^j_t(\vy_{s,k,j}) - \mF_t^{j}(\hat{\vx}_{s-1}) + \boldsymbol{\mu}_{s}^j\big)\|^2 \Big| \gF_{s,k, j-1}\Big]       \notag\\  
=&\;\E\Big[   \|\mF^j(\vy_{s, k, j}) - \mF^{j}(\hat{\vx}_{s-1}) - \big(  \mF^j_t(\vy_{s, k, j}) - \mF_t^{j}(\hat{\vx}_{s-1})\big)\|^2 \Big| \gF_{s,k, j-1}\Big]       \notag\\  
\le &\; \E\Big[  \|\mF_t^j(\vy_{s,k,j}) - \mF_t^{j}(\hat{\vx}_{s-1})\|^2 \Big| \gF_{s,k, j-1}\Big]   \notag\\ 
=\; & \frac{1}{n}\sum_{t'=1}^n \|\mF_{t'}^j(\vy_{s,k,j}) - \mF_{t'}^j(\hat{\vx}_{s-1})\|^2\notag \\
\le &\; \frac{2}{n}\sum_{t'=1}^n \Big(\|\mF_{t'}^j(\vy_{s,k,j}) -  \mF_{t'}^{j}({\vx}_{s,k})\|^2 + \|\mF_{t'}^j({\vx}_{s,k}) - \mF_{t'}^j(\hat{\vx}_{s-1})\|^2\Big), %\E\Big[ 2\Big( \|\mF_t^j(\vy_{s,k,j}) -  \mF_t^{j}({\vx}_{s,k})\|^2 + \|\mF_t^j({\vx}_{s,k}) - \mF_t^j(\hat{\vx}_{s-1})\|^2\Big)  \Big| \gF_{s,k, j-1}\Big],         
\label{eq:error-vr-3}   
% \le &\;  \E\Big[2 \Big((\vy_{s,k,j} - {\vx}_{s,k})^T\mQ^j (\vy_{s,k,j} - {\vx}_{s,k})+ \|\mF_t^j({\vx}_{s,k}) - \mF_t^j(\hat{\vx}_{s-1})\|^2\Big) \Big| \gF_{s,k, j-1} \Big]     \notag\\ 
% =& \; \E\Big[2 \Big((\vx_{s,k-1} - {\vx}_{s,k})^T\widehat{\mQ}^j (\vx_{s,k-1} - {\vx}_{s,k})+ \|\mF_t^j({\vx}_{s,k}) - \mF_t^j(\hat{\vx}_{s-1})\|^2\Big) \Big| \gF_{s,k, j-1} \Big], 
\end{align}
where the first inequality is by  $\E\big[\big(\mF_t^j(\vy_{s,k,j}) - \mF_t^{j}(\hat{\vx}_{s-1})\big)\big|\gF_{s,k, j-1}\big] = \mF^j(\vy_{s,k,j}) - \mF^{j}(\hat{\vx}_{s-1}),$ the second inequality is by  $(a+b)^2 \le 2(a^2 + b^2)$, $\forall a, b$.

Then similar to Eq.~\eqref{eq:error-vr1}, for Eq.~\eqref{eq:error-vr-3},  we have for all $t' \in \{1, \dots, n\},$
\begin{align}
\|\mF_{t'}^j(\vy_{s,k,j}) -  \mF_{t'}^{j}({\vx}_{s,k})\|^2      
 \le\;& (\vy_{s,k,j} - {\vx}_{s,k})^T\mQ^j (\vy_{s,k,j} - {\vx}_{s,k}) \notag\\  
  =\;& (\vx_{s,k-1} - {\vx}_{s,k})^T\widehat{\mQ}^j (\vx_{s,k-1} - {\vx}_{s,k}). \label{eq:error-vr3}
\end{align}

% the third inequality is by the smoothness assumption, and the last equality is by the definition of $\vy_{s,k,j}$ and $\widehat{\mQ}^j.$

% \cb{Write down a separate lemma for the variance reduction inequality? }

To bound the sums of Eqs.~\eqref{eq:error-vr1} and \eqref{eq:error-vr-3} over $j$ from $1$ to $m$, we use the Lipschitz constants  $L$ (defined in Assumption \ref{ass:Lip}) and  $\hat{L}$ (defined in Assumption \ref{assmpt:new-Lip}):
\begin{equation}
 \begin{aligned}
\sum_{j=1}^m  (\vx_{s,k-1} - \vx_{s,k-2})^T \widehat{\mQ}^j (\vx_{s,k-1} - \vx_{s,k-2})  \le&\; \hat{L}^2\|\vx_{s,k-1} - \vx_{s,k-2}\|^2, \\ 
\sum_{j=1}^m (\vx_{s,k-1} - {\vx}_{s,k})^T\widehat{\mQ}^j (\vx_{s,k-1} - {\vx}_{s,k}) \le&\; \hat{L}^2\|\vx_{s,k} - \vx_{s,k-1}\|^2, \\ 
\sum_{j=1}^m \frac{1}{n}\sum_{t' = 1}^n \|\mF_{t'}^j({\vx}_{s,k}) - \mF_{t'}^j(\hat{\vx}_{s-1})\|^2 =&\; \frac{1}{n} \sum_{t' = 1}^n\|\mF_{t'}({\vx}_{s,k}) - \mF_{t'}(\hat{\vx}_{s-1})\|^2 \\
\le&\; L^2 \| {\vx}_{s,k} - \hat{\vx}_{s-1}\|^2.  
\end{aligned}  \label{eq:error-vr-4} 
\end{equation}

To bound the terms from the fourth and fifth line, observe that for any fixed $\vu^j$, as $\vx_{s,k-1}^j\in\gF_{s,k, j-1}$, $\E[\mF_t^j(\vx_{s, k-1})|\gF_{s,k, j-1}] = \mF^j(\vx_{s, k-1})$ and $\E[\mF_t^{j}(\hat{\vx}_{s-1})|\gF_{s,k, j-1}] = \boldsymbol{\mu}_{s}^j$, % we have
\begin{align}
\E[\langle\mF_t^j(\vx_{s, k-1}) - \mF_t^j(\vy_{s, k-1, j}),&  \vx_{s,k-1}^j - \vu^j\rangle |\gF_{s,k, j-1}]  \nonumber\\
=\;&\langle\mF^j(\vx_{s, k-1}) - \mF^j(\vy_{s, k-1, j}),  \vx_{s,k-1}^j - \vu^j\rangle,    \label{eq:error-vr-5}   \\
\E[\langle  \mF^j(\vy_{s, k, j}) - \big(  \mF^j_t(\vy_{s,k,j}) \;&- \mF_t^{j}(\hat{\vx}_{s-1}) + \boldsymbol{\mu}_{s}^j\big) ,   \vx_{s,k-1}^j - \vu^j\rangle |\gF_{s,k, j-1}] = 0. \label{eq:error-vr-6}  
\end{align}

Hence, combining Lemma~\ref{lemma:coder-gap-change-vr} and Eqs.~\eqref{eq:error-vr}--\eqref{eq:error-vr-6}, we have
\begin{equation}
    \begin{aligned}
    &\E[E_{s,k}(\vu)|\gF_{s,k, j-1}]     \\
     \le&\; \sum_{j=1}^m\Big( a_s\E[\langle \mF^j(\vx_{s,k}) - \mF^j(\vy_{s, k, j}),  \vx_{s,k}^j - \vu^j\rangle|\gF_{s,k, j-1}]     \\ 
&\;\quad\quad- a_{s,k-1}\langle\mF^j(\vx_{s, k-1}) - \mF^j(\vy_{s, k-1, j}),  \vx_{s,k-1}^j - \vu^j\rangle\Big)     \\ 
&\; + \Big(a_s^2 L^2\alpha_2 - \frac{a_s \beta}{2}\Big)  \E[\| {\vx}_{s,k} - \hat{\vx}_{s-1}\|^2 |\gF_{s,k, j-1}]
+ \frac{a_{s,k-1}^2  \hat{L}^{2} \alpha_1}{2} \|\vx_{s,k-1} - \vx_{s,k-2}\|^2  \\
&  + \Big( a_s^2\hat{L}^2\alpha_2   + \frac{1}{2\alpha_1} + \frac{1}{2\alpha_2} + \frac{a_s \beta}{2}  -  \frac{K(1+A_{s-1}\gamma)}{2}\Big)\E[ \| {\vx}_{s,k} - \vx_{s,k-1}\|^2 |\gF_{s,k, j-1}]   \\
&\; - \frac{a_s \beta}{2}\big( \| \vx_{s, k-1} - \vu\|^2 - \|\hat{\vx}_{s-1} -  \vu\|^2\big)- \frac{a_s \gamma}{2}\E[\|\vx_{s,k} - \vu\|^2|\gF_{s,k, j-1}]. 
    \end{aligned}\label{eq:error-vr-5-1} 
\end{equation}

To complete the proof, it remains to choose the points $\vxh_s,$ the step sizes $a_s, a_{s, k-1}$, and parameters $\alpha_1, \alpha_2,$ and $\beta.$ First, by our choice of step sizes, we have $a_{s, 0} = a_{s-1}$, $a_{s,1} = \cdots = a_{s,K} = a _s$, which makes the terms from the first two lines of the right-hand side of Eq.~\eqref{eq:error-vr-5-1} telescope. 

 %First, we 
 Next, we define $\vxh_{s, k}$ by  $\vxh_{s,k} = \frac{\beta }{\beta+\gamma}\vx_{s, k-1} + \frac{\gamma}{\beta+\gamma}\vx_{s,k}$ so that $\vxh_s = \frac{1}{K}\sum_{k=1}^K \vxh_{s, k}.$ As a consequence, by Young's inequality,
\begin{align}
\beta \| \vx_{s, k-1} - \vu\|^2 + \gamma  \|\vx_{s,k} - \vu\|^2  \ge (\beta+\gamma)\|\vxh_{s,k}-\vu\|^2.
\end{align}
This will make the terms in the last line of Eq.~\eqref{eq:error-vr-5-1} telescope, after summing over $k \in \{1, \dots, K\}$. 

To make the remaining terms in Eq.~\eqref{eq:error-vr-5-1} either telescope or cancel out, we make the following step size and parameter choices: 
%Then by using the following settings: 
$\forall s \ge 1,$
\begin{align}
&\; \alpha_1 = \frac{8}{K(1+A_{s-1}\gamma)}, \quad \alpha_2 = \frac{8}{K(1+A_{s-1}\gamma)}, \quad \beta =  \frac{2L}{\sqrt{K}},    \nonumber\\
&\; a_s \le (1+A_{s-1}\gamma)\min\Big\{ \frac{\sqrt{K}}{8L},  \frac{K}{8\hat{L}}\Big\},\nonumber\\
&\; a_{s+1} \le \big(1 + \frac{\gamma}{\beta}\big)a_s. \nonumber
\end{align}
Under this choice, using that $A_s$ is non-decreasing with $s$ and $K \geq 1,$ it is not hard to verify that 
\begin{equation}
\begin{aligned}
{a_s}^2\alpha_2 L^2 - \frac{a_s \beta}{2} \le&\; a_s\Big( \frac{L}{\sqrt{K}}  - \frac{L}{\sqrt{K}}\Big) = 0, \\ 
\frac{{a_{s,k-1}}^2  \hat{L}^{2} \alpha_1}{2}\le&\; \frac{K(1+A_{s-1}\gamma)}{16 },   \\
{a_s}^2\hat{L}^2 \alpha_2   + \frac{1}{2\alpha_1} + \frac{1}{\alpha_2} + \frac{a_s \beta}{2}  -  \frac{K(1+A_{s-1}\gamma)}{2} \le&\; - \frac{K(1+A_{s-1}\gamma)}{8},    \\
a_{s+1}\beta  \le &\;  a_s(\beta + \gamma).      
\end{aligned}  \label{eq:error-vr-61}  
\end{equation}

By denoting $A_{s, 0} = A_{s-2}, A_{s,1}=A_{s,2}=\cdots =A_{s,K} = A_{s-1}$ and combining Eqs.~\eqref{eq:error-vr-5}--\eqref{eq:error-vr-61}, we have: $\forall k \ge 1,$
\begin{equation}
\begin{aligned}
\E[E_{s,k}(\vu)|\gF_{s,k, j-1}]     
 \le&\; \sum_{j=1}^m\Big( a_s\E[\langle \mF^j(\vx_{s,k}) - \mF^j(\vy_{s, k, j}),  \vx_{s,k}^j - \vu^j\rangle|\gF_{s,k, j-1}]     \\ 
&\;\quad\quad- a_{s,k-1}\langle\mF^j(\vx_{s, k-1}) - \mF^j(\vy_{s, k-1, j}),  \vx_{s,k-1}^j - \vu^j\rangle\Big)     \\ 
&\;+  \frac{K(1+A_{s,k-1}\gamma)}{8}    \|\vx_{s,k-1} - \vx_{s,k-2}\|^2  \\ 
&\;-  \frac{K(1+A_{s-1}\gamma)}{8} \E[ \| {\vx}_{s,k} - \vx_{s,k-1}\|^2 |\gF_{s,k, j-1}]   \\
&\;+\beta a_s\|\hat{\vx}_{s-1}-\vu\|^2 - \beta a_{s+1}\E[\|\vxh_{s,k} - \vu\|^2|\gF_{s,k, j-1}]. 
\end{aligned}\label{eq:error-vr-7} 
\vspace{-4mm}
\end{equation}
% Let $\vx_{s,K} = \vx_{s+1,0},\vx_{s,K-1} = \vx_{s+1,-1}$. 
% \cb{How to say this fact? In algorithm? }  
Taking expectation w.r.t.~all the randomness in the algorithm, using the tower property of expectation $\E[\E[\cdot|\gF]] = \E[\cdot]$, and summing  Eq.~\eqref{eq:error-vr-7}  over $k\in \{1,\dots, K\}$,  we have \vspace{-4mm}
\begin{align}
&\;\sum_{k=1}^K \E\Big[E_{s,k}(\vu)\Big]% \Big| \gF_{s, k-1 }\Big]
\nonumber\\
\le&\;  \sum_{j=1}^m\Big( a_s\E\big[\langle \mF^j(\vx_{s,K}) - \mF^j(\vy_{s, K, j}),  \vx_{s,K}^j - \vu^j \rangle %| \gF_{s,k, j-1} ]  
\nonumber   \\ 
&\;\quad\quad- a_{s,0}\langle\mF^j(\vx_{s, 0}) - \mF^j(\vy_{s, 0, j}),  \vx_{s,0}^j - \vu^j\rangle\big]\Big)  \nonumber\\ 
&\;  + \frac{K(1+A_{s,0}\gamma)}{8}   \E[ \|\vx_{s,0} - \vx_{s,-1}\|^2 ]
-  \frac{K(1+A_{s-1}\gamma)}{8} \E[ \| {\vx}_{s,K} - \vx_{s,K-1}\|^2 ]%|\gF_{s,k, j-1}] 
\nonumber\\ 
&\; - \sum_{k=1}^K \frac{\beta\alpha_{s+1}}{2}\E[\|\vxh_{s,k} - \vu\|^2 ]%| \gF_{s, 0}]  
+ \frac{ \beta K \alpha_s}{2}\E[\|\hat{\vx}_{s-1} -\vu\|^2] \nonumber \\
\le &\;   \sum_{j=1}^m\Big( a_s\E\big[\langle \mF^j(\vx_{s+1,0}) - \mF^j(\vy_{s+1, 0, j}),  \vx_{s+1,0}^j - \vu^j\rangle %| \gF_{s,0}  ]  
\nonumber   \\ 
&\;\quad\quad- a_{s-1}\langle\mF^j(\vx_{s, 0}) - \mF^j(\vy_{s, 0, j}),  \vx_{s,0}^j - \vu^j\rangle\big]\Big)  \nonumber\\  
&\;  + \frac{K(1+A_{s-2}\gamma)}{8}  \E[  \|\vx_{s,0} - \vx_{s,-1}\|]^2 
-  \frac{K(1+A_{s-1}\gamma)}{8} \E[ \| {\vx}_{s+1,0} - \vx_{s+1,-1}\|^2] %|\gF_{s,0}]       
\nonumber\\ 
&\;  + \frac{\beta K a_{s}}{2}\E[\|\hat{\vx}_{s-1} - \vu\|^2] - \frac{\beta K a_{s+1}}{2}\E[\|\hat{\vx}_s - \vu\|^2],\label{eq:error-vr-8} 
\end{align}
where the first inequality is by our setting $a_{s, 0} = a_{s-1}$, $a_{s,1} = \cdots = a_{s,K} = a _s$
and $A_{s,0}=A_{s-2}, A_{s,1} = \cdots = A_{s,K} = A_{s-1},$ the second inequality is by our definitions $\vx_{s,K}  = \vx_{s+1,0}, \vx_{s,K-1}  = \vx_{s+1,-1},  \vy_{s, K, j} = \vy_{s+1, 0, j}$ and  $ \hat{\vx}_s=\frac{1}{K}\sum_{k=1}^K\vxh_{s,k} = \frac{1}{K}\sum_{k=1}^K\big( \frac{\beta }{\beta+\gamma}\vx_{s, k-1} + \frac{\gamma}{\beta+\gamma}\vx_{s,k}\big)$, and the convexity of $\|\cdot\|^2.$  

Finally, summing Eq.~\eqref{eq:error-vr-8} over $s \in \{ 1,\dots, S\},$ we have
\begin{align}
 &\;\sum_{s=1}^S\sum_{k=1}^K \E\Big[E_{s,k}(\vu) \Big]\nonumber\\   
 \le\;&\sum_{j=1}^m\Big( a_S\E[\langle \mF^j(\vx_{S+1,0}) - \mF^j(\vy_{S+1, 0, j}),  \vx_{S+1,0}^j - \vu^j\rangle]  \nonumber   \\ 
&\;\quad\quad- a_{0}\langle\mF^j(\vx_{1, 0}) - \mF^j(\vy_{1, 0, j}),  \vx_{1,0}^j - \vu^j\rangle\Big)  \nonumber\\  
&\;  + \frac{K(1+A_{-1}\gamma)}{8}    \|\vx_{1,0} - \vx_{1,-1}\|^2 
-  \frac{K(1+A_{S-1}\gamma)}{8} \E[ \| {\vx}_{S+1,0} - \vx_{S+1,-1}\|^2 ]       \nonumber\\ 
&\;\quad\quad - \frac{\beta K a_{S+1} }{2}\|\hat{\vx}_S- \vu\|^2 + \frac{\beta K a_{1}}{2}\|\hat{\vx}_{0} - \vu\|^2\nonumber\\  
\le\;&\sum_{j=1}^m  a_S\E[\langle \mF^j(\vx_{S+1,0}) - \mF^j(\vy_{S+1, 0, j}),  \vx_{S+1,0}^j - \vu^j\rangle]  \nonumber   \\ 
&\; -  \frac{K(1+A_{S-1}\gamma)}{8} \E[ \| {\vx}_{S+1,0} - \vx_{S+1,-1}\|^2 ]  \nonumber   \\
&\;\quad\quad - \frac{\beta K a_{S+1} }{2}\E[\|\hat{\vx}_S- \vu\|^2] + \frac{\beta K a_{1}}{2}\|\vx_{0} - \vu\|^2,
\end{align}
where the last inequality is by $a_0 = 0$ and $\vx_{1,0} = \vx_{1,-1} = \vx_0 = \vxh_0.$ 
\end{proof}

%\subsection{Proof of Theorem \ref{thm:main-coder-vr}}

\end{document}